\documentclass[reqno,12pt]{amsart}
\usepackage[english]{babel}
\usepackage{amsmath}

\usepackage{amssymb}
\usepackage{graphicx}
\usepackage{epstopdf}

\usepackage{subfigure}
\usepackage{enumitem}

\usepackage{amsaddr}

\newcommand{\bpr}{\begin{trivlist} \item[]{\bf Proof. }}
\newcommand{\epr}{\hspace*{\fill} $\qed$\end{trivlist}}

\newcommand{\e}{{\mathrm e}}

\newcommand{\be}{\begin{eqnarray}}
\newcommand{\ee}{\end{eqnarray}}
\newcommand{\ba}{\begin{align}}
\newcommand{\ea}{\end{align}}
\newcommand{\bi}{\begin{itemize}}
\newcommand{\ei}{\end{itemize}}

\newcommand{\eqlab}[1]{\label{eq:#1}}
\renewcommand{\eqref}[1]{(\ref{eq:#1})}

\newcommand{\figref}[1]{Fig.~\ref{fig:#1}}
\newcommand{\figlab}[1]{\label{fig:#1}}

\newcommand{\lemmaref}[1]{Lemma~\ref{lemma:#1}}
\newcommand{\lemmalab}[1]{\label{lemma:#1}}
\newcommand{\remref}[1]{Remark~\ref{remark:#1}}
\newcommand{\remlab}[1]{\label{remark:#1}}

\newtheorem{theorem}{Theorem}[section]
\newtheorem{proposition}[theorem]{Proposition}

\newtheorem{lemma}[theorem]{Lemma}
\newtheorem{cor}[theorem]{Corollary}

\newtheorem{remark}[theorem]{Remark}

\numberwithin{equation}{section}
\usepackage{fullpage}
\usepackage{color,comment,xcolor}
\definecolor{orange}{RGB}{255,127,0}

\newcommand\KUK[1]{{\color{black}{#1}}} 
\newcommand\RH[1]{{\color{black}{#1}}}
\newcommand\NEW[1]{{\color{black}{#1}}}

\begin{document}
\title{On entry-exit formulas for degenerate turning point problems in planar slow-fast systems}
\author{R. Huzak}
\address{Hasselt University, Campus Diepenbeek, Agoralaan Gebouw D, 3590 Diepenbeek, Belgium}
\author{K. Uldall Kristiansen}
\address{Department of Applied Mathematics and Computer Science, 
Technical University of Denmark, 
2800 Kgs. Lyngby, 
Denmark }


 \begin{abstract}
\KUK{In this paper, we study degenerate entry-exit problems associated with planar slow-fast systems having an invariant line $\{(x,y)\,:\,y=0\}$ with a turning point at $x=0$. The degeneracy stems from the fact that the slow flow has a saddle-node of even order $2n$, $n\in \mathbb N$, at the turning point, i.e. $x' = -x^{2n}(1+o(1))$ for $\epsilon=0$. We are motivated by the appearance of such turning point problems (for $n=1$) in the graphics $(I_2^1)$ and $(I_4^1)$, through a nilpotent saddle-node singularity at infinity, in the Dumortier-Roussarie-Rousseau program (for solving the finiteness part of Hilbert’s 16th problem for quadratic polynomial systems).  Our results show, under additional hypothesis, that in the case $n=1$ there is a well-defined entry-exit relation for $\epsilon\to 0$. The associated Dulac map is smooth w.r.t. $(\epsilon,\epsilon \log \epsilon^{-1})$. On the other hand for the cases $n\ge 2$, we show that the entry-exit relation requires additional control parameters. Our approach follows the one used by De Maesschalck, P. and Schecter, S. (JDE 2016) for a different type of degenerate entry-exit problem. In particular, we apply blow-up {after} having first performed a singular coordinate transformation of $y$. The degeneracy at $x=0$ requires an additional blow-up. We finally apply the result for $n=1$ to a normal form for the unfolding of the \NEW{relevant} graphics in the Dumortier-Roussarie-Rousseau program. Here we also demonstrate that the singular transformation of $y$ due to De Maesschalck, P. and Schecter, S. (JDE 2016) has practical significance in numerical computations.  }

\smallskip 

\smallskip 

\noindent \textbf{keywords.} Entry-exit, GSPT, blowup, the Dumortier-Roussarie-Rousseau
program, Hilbert's 16th problem.

 \end{abstract}

 \bigskip
\smallskip

 \maketitle
 \section{Introduction}
 
Consider a planar slow-fast system
\begin{equation}\eqlab{model0}
 \begin{aligned}
  \dot x &= \epsilon f(x,y,\epsilon),\\
  \dot y &= y h(x,y,\epsilon),
 \end{aligned}
 \end{equation}
where $\epsilon\ge 0$ is a singular perturbation parameter kept small and $f$ and $h$ are ($C^\infty$-)smooth functions. 
We suppose the following:
\begin{align}\eqlab{cond1}
x h(x,0,0)<0\quad \forall\, x\in I\backslash \{0\}, 
\end{align}
where $I\subset \mathbb R$ is a compact interval.
Then $x=0$ is a turning point where the stability of the invariant line $\{ y=0\} $ changes from normally attracting for $x>0$ to normally repelling for $x<0$. 

If $f(x,0,0)<0$ for all $x\in I$, then we deal with the well-known entry-exit problem studied by many authors (see \cite{DMS2016,Hsu2017,Kolesov,Sch1985} and references therein). The entry-exit problem consists of describing the transition map $\Sigma_{\text{in}}\rightarrow \Sigma_{\text{out}}$, $(x_{\text{in}},\delta)\mapsto (\Delta(x_{\text{in}},\epsilon),\delta)$ where 
\begin{align*}
\Sigma_{\text{in}}\,&:\,y=\delta, \,x_{\text{in}}\in I_{\text{in}},\\
\Sigma_{\text{out}}\,&:\, y=\delta,\,x_{\text{out}}\in I_{\text{out}},
\end{align*}
with $\delta>0$ small, $I_{\text{in}}\subset I\cap (0,\infty)$  and $I_{\text{out}}\subset I\cap (-\infty,0)$ compact intervals. It is well known that the problem with the intervals $I_{\text{in}}$ and $I_{\text{out}}$ is well defined for all $0<\epsilon\ll 1$ if for every $x_{\text{in}}\in I_{\text{in}}$, there is an $x_{\text{out}}=\Delta_0(x_{\text{in}})\in I_{\text{out}}$ such that 
\begin{align}
 \int_{x_{\text{in}}}^{x_{\text{out}}} \frac{h(s,0,0)}{f(s,0,0)}ds = 0.\eqlab{Delta0}
\end{align}
In particular, the Dulac map $\Delta(\cdot,\epsilon):I_{\text{in}}\rightarrow \mathbb R$, $\epsilon\in ]0,\epsilon_0[)]$, takes the following form
\begin{align*}
 \Delta(x_{\text{in}},\epsilon) = \Delta_0(x_{\text{in}}) + \NEW{\mathcal O(\epsilon)},
\end{align*}
where $\NEW{\mathcal O(\epsilon)}$ is a smooth function of $(x_{\text{in}},\NEW{\epsilon})$ and is identically zero when $\epsilon=0$. \NEW{Notice that the trajectory leaves $y=0$ after it has become unstable at $x=0$. This phenomena is also known as Pontryagin delay or bifurcation delay, see \cite{MR1167001,Kolesov}.}

On the other hand, if there is an $x_{\text{in}}\in I_{\text{in}}$ so that there is no $x_{\text{out}}\in I_{\text{out}}$ satisfying \eqref{Delta0}, then the entry-exit problem  with the intervals $I_{\text{in}}$ and $I_{\text{out}}$ is not well-defined for all $\epsilon>0$ sufficiently small.
For details, we refer to \cite{DMS2016} where the entry-exit problem is studied using a novel blow-up  technique
\KUK{based upon writing the equations in terms of $(x,z)$ where $z$ is related to $y$ through
\begin{align}\nonumber
    y = \e^{-z^{-1}},
\end{align}
see \cite[Corollary 1.2]{DMS2016}.
We will follow this approach in the present paper, see Section \ref{section-blow-up} for further details.}
The entry-exit formula \eqref{Delta0} plays an important role in the study of relaxation oscillations in predator-prey systems (see, e.g. \cite{AiYi24,Hsu19,Yao24} and references therein). We also refer to \cite{EE-Nikola,EE-Zhang} for some other more degenerate entry-exit problems and applications in $\mathbb R^2$ and $\mathbb R^3$.


%
More precisely, if we assume $h'_x(0,0,0)<0$ and use \eqref{cond1}, then, up to smooth equivalence, system \eqref{model0} has the following form 
\begin{equation}\eqlab{model0NF}
\begin{aligned}
 \dot x &= \epsilon f_0(x,\epsilon)  + y g_0(x,y,\epsilon),\\
 \dot y &=-xy,
\end{aligned}
\end{equation}
 for some smooth functions $f_0$ and $g_0$ and $g_0(0,0,0)=0$ (see Lemma \ref{lemma-NF} in Section \ref{subsection-NF}). Here, $h'_x$ denotes the partial derivative of $h$ w.r.t. $x$. 
  In this paper, we focus on the system \eqref{model0NF}
 and assume that $g_0$ is an arbitrary smooth function ($g_0(0,0,0)$ is not necessarily zero) and that $f_0(x,0)$ has a zero of even multiplicity  $2n$, $n\in \mathbb N$, at $x=0$:
\begin{align*}
 \frac{\partial^{k} f_0}{\partial x^k}(0,0) = 0\quad \forall\,k\in \{0,1,\ldots,2n-1\},\quad  \frac{\partial^{2n} f_0}{\partial x^{2n}}(0,0) < 0.
\end{align*} 
Since the multiplicity is even, the entry-exit problem may still be well defined.


\begin{remark}
   When $f_0(x,0)<0$, the entry-exit problem associated with \eqref{model0NF}, with an arbitrary smooth function $g_0$, has been studied in \cite[Section 5]{Hsu19}.
\end{remark}

\begin{figure}[htb]
	\begin{center}
		\includegraphics[width=5.9cm,height=4.1cm]{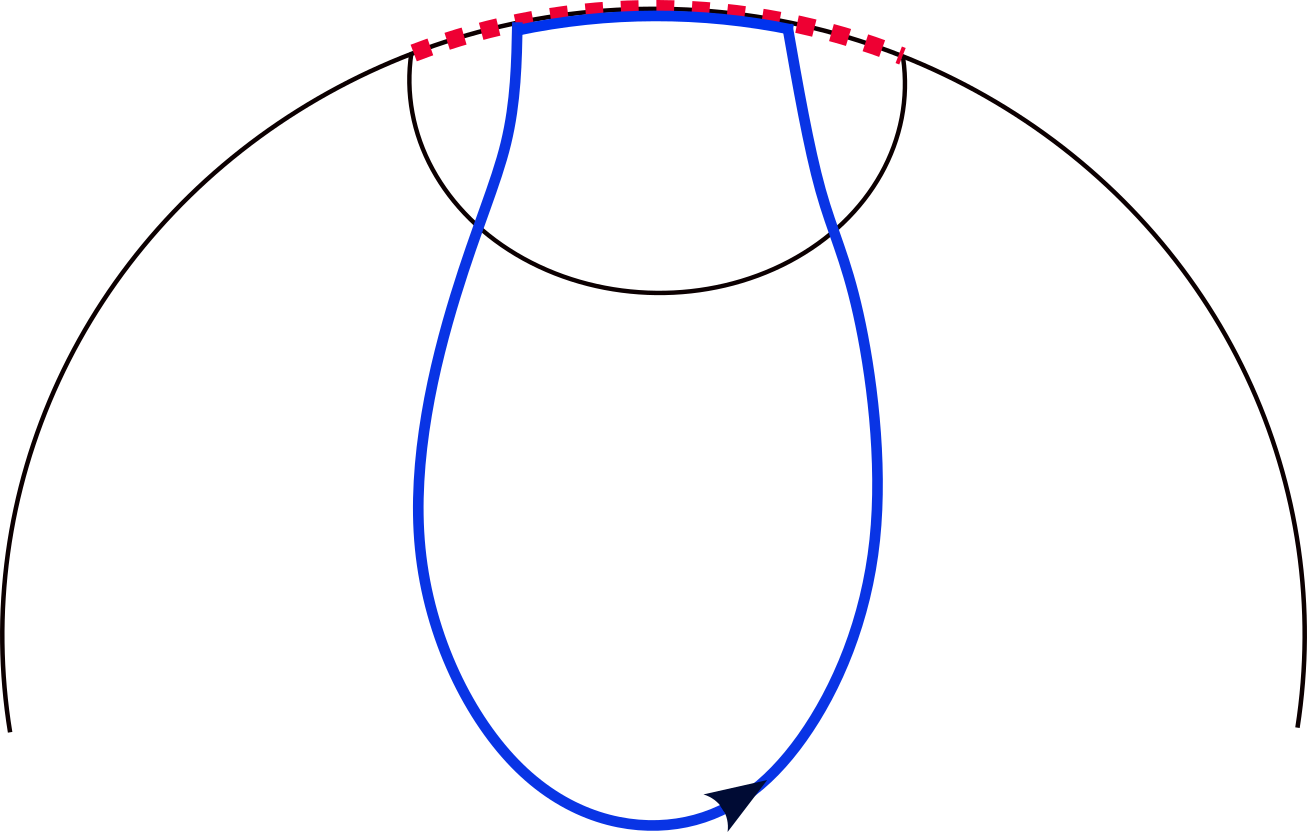}
         \end{center}
 \caption{A limit periodic set after desingularization of the graphic $(I_2^1)$ through a nilpotent saddle-node at infinity in the Dumortier-Roussarie-Rousseau program.}
	\label{fig-Motivation}
\end{figure}

The case $n=1$ is relevant to the analysis of the graphics $(I_2^1)$ and $(I_4^1)$ through a nilpotent saddle-node singularity at infinity in the Dumortier-Roussarie-Rousseau program (see \cite[Figure 8]{Program}). The main goal of this program is to solve the finiteness part of Hilbert’s 16th problem for quadratic polynomial systems. After a blow-up at the singular point at infinity one can detect all possible limit periodic sets related to $(I_2^1)$ and $(I_4^1)$ whose
finite cyclicity needs to be studied. Such a limit periodic set is given in Fig. \ref{fig-Motivation}. Here the invariant line $\{y=0\}$ corresponds to infinity in a Poincar\'e compactification and one needs to deal with the entry-exit problem \eqref{model0NF} where $f_0(x,0)$ has a zero of multiplicity $2$ and $g_0\ne 0$. For more details, \RH{we refer to Section \ref{example} and \cite{HuzRou202.}.} The case $n\ge 2$ is similarly relevant for the general version of Hilbert's 16th problem. \RH{Besides this purely mathematical question, we believe that the entry-exit problem treated in this paper could also be important when one studies relaxation oscillations in predator-prey systems and other applied slow-fast models.}

The upper bounds for the number of canard limit cycles of \eqref{model0NF} with $g_0(0,0,0)\ne 0$ (and more general slow-fast systems without presence of an
invariant line) have been studied in \cite{detectable,detectable-1} using the notion of the slow divergence integral \cite[Chapter 5]{DDR-book-SF}. We point out that the entry-exit problem of \eqref{model0NF} has not been treated in \cite{detectable,detectable-1}.

The paper is organized as follows. In Section \ref{section-results} we define our slow-fast model and state the main results. We introduce a blow-up in Section \ref{section-blow-up}. In Sections \ref{section-proof-n=1} and \ref{proof-n>2} we prove our main results. Section \ref{example} is devoted to \KUK{the entry-exit problem of the graphics $(I_2^1)$ and $(I_4^1)$}. \KUK{Here we also illustrate our results further by performing some numerical computations.}

\section{Slow-fast model and statement of results}\label{section-results}

\subsection{Normal form}\label{subsection-NF}
We start this section with the following lemma. 
\begin{lemma}
    \label{lemma-NF}
    Consider system \eqref{model0} and assume that $h'_x(0,0,0)<0$ and \eqref{cond1} are satisfied. Then there exists a smooth $\epsilon$-family of coordinate changes bringing \eqref{model0}, near $I\times\{0\}\subset \mathbb R^2$, in \eqref{model0NF}, up to multiplication by a smooth positive function.
\end{lemma}

\begin{proof}
The conditions $h'_x(0,0,0)<0$ and \eqref{cond1} imply that,  after an $\epsilon$-dependent shift of $x$, we can write $h$ in \eqref{model0} as 
\begin{align*}
 h(x,y,\epsilon) = xh_0(x,\epsilon)+y h_1(x,y,\epsilon),
\end{align*}
where
\begin{align*}
 h_0(x,0)\le -c<0\quad \forall\,x\in I,
\end{align*}
for some $c>0$ small enough.
Upon dividing the right-hand side by the positive factor $-h_0(x,\epsilon)$ we can achieve that $h_0(x,\epsilon)=-1$ for all $x\in I$ and $\epsilon\ge 0$ small enough. Now define a new coordinate
\begin{align*}
 \tilde x = x -y h_1(x,y,\epsilon).
\end{align*}
By the implicit function theorem this induces a smooth $\epsilon$-family of coordinate changes $(x,y) \mapsto (\tilde x,y)$ for $x\in I$ and $y$ kept close to zero. After applying the coordinate change we obtain \eqref{model0NF} 
with smooth functions $f_0$ and $g_0$ and $g_0(0,0,0)=0$ (we drop the tildes).
\end{proof}

In this paper, we will consider an arbitrary smooth function $g_0$ and a generic unfolding $f_\lambda$ of $f_0$ in \eqref{model0NF} :
\begin{align}
 f_\lambda(x,\epsilon)  = \lambda_0 + \lambda_1 x +\cdots+\lambda_{2n-1} x^{2n-1} + x^{2n} \zeta_{2n}(x,\epsilon),\eqlab{flambda0}
\end{align}
where $\lambda=(\lambda_0,\ldots,\lambda_{2n-1})$ are the unfolding parameters kept close to zero, $\zeta_{2n}$ is a smooth function and $\zeta_{2n}(x,0)<0$ for all $x\in I$. By redefining $\epsilon$, we can easily achieve that 
$$\zeta_{2n}(0,0)=-1.$$

It is natural to consider a blow-up  of parameters
\begin{align}\eqlab{blowuppar}
 \epsilon = r\bar \epsilon,\quad \lambda_i=r^{2n-i} \overline{\lambda}_i,\quad i\in \{0,\ldots,2n-1\}, \ (\bar\epsilon,\overline{\lambda}_0,\ldots,\overline{\lambda}_{2n-1})\in\mathbb S^{2n}, \ \bar\epsilon\ge 0.
\end{align}
We will only focus on the single chart $\bar \epsilon=1$, setting 
\begin{align*} \lambda_i=\epsilon^{2n-i} \widetilde{\lambda}_{i},\quad i\in \{0,\ldots,2n-1\},
\end{align*}
with $\widetilde\lambda=(\widetilde{\lambda}_{0},\ldots,\widetilde{\lambda}_{2n-1})$ kept in a compact subset $\Lambda$ of $\mathbb R^{2n}$,
 so that 
\begin{align}\eqlab{flambda}
 f_{\widetilde\lambda}(x,\epsilon):=f_{(\epsilon^{2n}\widetilde\lambda_0,\ldots,\epsilon\widetilde\lambda_{2n-1})}(x,\epsilon)  = \epsilon^{2n} P_{\widetilde\lambda}(\epsilon^{-1} x) + x^{2n}(\zeta_{2n}(x,\epsilon)+1),
\end{align}
with
\begin{align}\eqlab{P2n}
 P_{\widetilde\lambda}(x_2) :=\widetilde\lambda_0 +  \widetilde \lambda_1 x_2 +\cdots+ \widetilde\lambda_{2n-1} x_2^{2n-1} - x_2^{2n},
\end{align}
For simplicity we drop the tildes and write $f_{\lambda}, P_{\lambda}$ instead of $f_{\widetilde\lambda},P_{\widetilde\lambda}$.

Finally, we consider $n\in \mathbb N$ and 
\begin{equation}\eqlab{xy}
\begin{aligned}
 \dot x &= \epsilon f_\lambda(x,\epsilon) + y g(x,y,\epsilon),\\
 \dot y &= - x y,
\end{aligned}
\end{equation}
with $f_\lambda$ given by \eqref{flambda} with $\lambda\in \Lambda$ and $\zeta_{2n}(0,0)=-1$. We will suppose that $f_\lambda$ and $g$ are $C^\infty$-smooth functions.  

The following lemma provides conditions for the existence of a regular passage along the invariant line $\{y=0\}$ of \eqref{xy}, for $\epsilon>0$. 

\begin{lemma}   \label{lemma-condition}
Suppose that
\begin{align}\eqlab{f2n}
 \zeta_{2n}(x,0)\le -c,\quad \forall\, x\in I,
\end{align} 
and that
\begin{align}\eqlab{P2ncond}
 P_\lambda(x_2)\le -c\quad \forall\,x_2\in \mathbb R, \ \lambda\in \Lambda,
\end{align}
for some $c>0$ and $\Lambda\subset \mathbb R^{2n}$ a compact set.
Then there is an $\epsilon_0>0$ such that 
\begin{align*}
f_\lambda(x,\epsilon)<0\quad \forall\,  x\in I, \,\epsilon\in ]0,\epsilon_0[, \ \lambda\in \Lambda.
\end{align*}
\end{lemma}
\begin{proof}
Let $\gamma>0$ and $\epsilon>0$. 

For $x\in I\backslash [-\epsilon \gamma^{-1},\epsilon \gamma^{-1}]$, we use \eqref{flambda} and estimate:
\begin{align*}
 f_\lambda(x,\epsilon)  &= x^{2n} \left(\epsilon^{2n} x^{-2n}\lambda_0 +  \epsilon^{2n-1}  x^{-2n+1} \lambda_1  +\cdots+ \epsilon x^{-1} \lambda_{2n-1}\right) + x^{2n} \zeta_{2n}(x,\epsilon)\\
 &<x^{2n}\left(\gamma^{2n} \vert \lambda_0\vert +  \gamma^{2n-1}  \vert \lambda_1\vert  +\cdots+\gamma \vert \lambda_{2n-1}\vert+\zeta_{2n}(x,\epsilon)\right)\\
 &\le x^{2n} \left(\mathcal O(\gamma) -C\right)<0,
 \end{align*}
for a constant $C>0$, for all $\gamma>0$ small enough, $\lambda\in \Lambda$ and for all $\epsilon\in]0,\epsilon_0[$, with $\epsilon_0>0$ small. Here we have used \eqref{f2n}. We fix such a $\gamma>0$.

Using \eqref{flambda} again, we can write
\begin{align*}
 f_\lambda(x,\epsilon) = \epsilon^{2n} \left(P_\lambda(\epsilon^{-1} x) +\left(\frac{x}{\epsilon}\right)^{2n} (\zeta_{2n}(x,\epsilon)+1)\right),
\end{align*}
for $\epsilon>0$. This, together with \eqref{P2ncond} and $\zeta_{2n}(0,0)=-1$, implies that $f_\lambda(x,\epsilon) <0$ for all $x\in [-\epsilon \gamma^{-1},\epsilon \gamma^{-1}]$, $\lambda\in \Lambda$ and all $\epsilon\in ]0,\epsilon_0[$, up to shrinking $\epsilon_0$ if necessary. This completes the proof of the lemma.
\end{proof}
It is clear from \eqref{P2n} that $\Lambda$ with the property \eqref{P2ncond} exists. For the rest of the paper, we assume that \eqref{f2n} and \eqref{P2ncond} are satisfied.

\begin{remark}
     \KUK{For a complete analysis of the unfolding \eqref{flambda0}, one would have to study the remaining charts associated with the blow-up \eqref{blowuppar}. We leave this to the interested reader, but we believe that these cases can be covered through a combination of the present work (in the chart $\bar \epsilon=1$) with standard results on entry-exit, e.g. \cite{DMS2016}. }
     %
\end{remark}

\subsection{Statement of the main results}\label{subsection-statement} In this section, we state the main results (Theorem \ref{Thm-n=1} for $n=1$ and Theorem \ref{thm-n>1} for $n\ge 2$). 

For $\epsilon=0$, system \eqref{xy} becomes
\begin{align*}
 \dot x &=  y g(x,y,0),\\
 \dot y &= -x y,
\end{align*}
or written as \KUK{as an equation or $y=y(x)$} for $y\ne 0$
\begin{align}\eqlab{dydx}
 \frac{dx}{dy} = -x^{-1} g(x,y,0).
\end{align}
Notice that \eqref{dydx} is well defined and regular for all $y\in [0,\delta]$, $\delta\in (0,1)$ small enough, and $x\in I_{\text{in}}$ where $I_{\text{in}}\subset (0,\infty)$ is a compact interval. More precisely, let $x= \psi(x_{\text{in}},y),\,y\in [0,\delta]$, denote the $C^\infty$-smooth solution of \eqref{dydx} with the initial condition $\psi(x_{\text{in}},\delta)=x_{\text{in}}\in I_{\text{in}}$ at $y=\delta$. We then define $(x_{\text{in}}^b,0)$ as the base point of $(x_{\text{in}},\delta)\in \Sigma_{\text{in}}$ on $y=0$:
\begin{align}
 x_{\text{in}}^b :=\psi(x_{\text{in}},0).\eqlab{xbx0}
\end{align} We suppose that $x_{\text{in}}^b\in I\cap (0,\infty)$, with $I$ fixed in Lemma \ref{lemma-condition}. Finally, $x_{\text{out}}^b=x_{\text{out}}^b(x_{\text{out}})\in I\cap (-\infty,0)$ is defined completely analogously  for $x_{\text{out}}\in I_{\text{out}}$, see Fig. \ref{fig-entry-exit}. 

\begin{figure}[htb]
	\begin{center}
		\includegraphics[width=6.7cm,height=3.5cm]{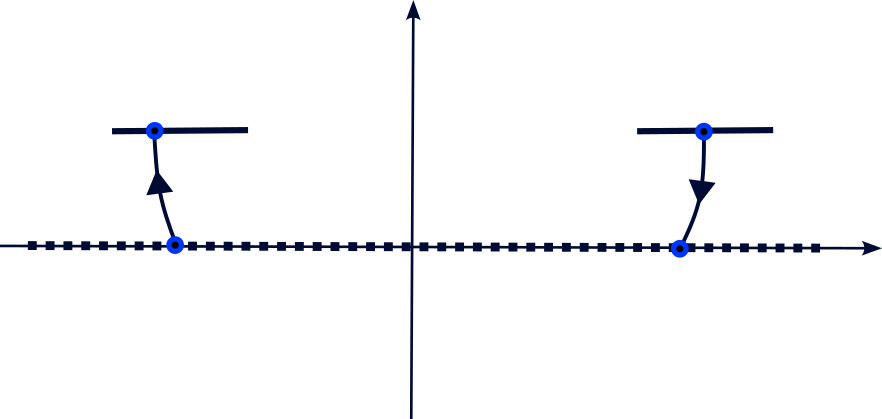}
  {\scriptsize
\put(4,40){$x$}
\put(-97,95){$y$}
\put(-55,74){$(x_{\text{in}},\delta)$}
\put(-60,27){$(x_{\text{in}}^b,0)$}
\put(-169,27){$(x_{\text{out}}^b,0)$}
\put(-175,74){$(x_{\text{out}},\delta)$}
\put(-20,66){$\Sigma_{\text{in}}$}
\put(-133,66){$\Sigma_{\text{out}}$}}
\caption{Illustration of the base point $(x_{\text{in}/\text{out}}^b,0)$ of $(x_{\text{in}/\text{out}},\delta)\in \Sigma_{\text{in}/\text{out}}$.}
	\label{fig-entry-exit}
    \end{center}
\end{figure}
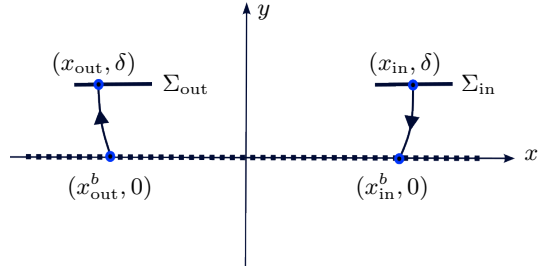

\RH{First, we assume that $n=1$ in \eqref{xy}. We then consider the Cauchy principal values
\begin{align}\eqlab{CauchyPV}
\text{p.v.} \int_{x_{\text{out}}^b}^{x_{\text{in}}^b}\frac{1}{s\zeta_2(s,0)}ds:= \lim_{\rho\to 0^+}\left(\int_{x_{\text{out}}^b}^{-\rho}\frac{1}{s\zeta_2(s,0)}ds+ \int_{\rho}^{x_{\text{in}}^b}\frac{1}{s\zeta_2(s,0)}ds\right)
\end{align}
and 
\begin{align}\eqlab{CauchyPV-second-part}
 \text{p.v.} \int_{-\infty}^{+\infty}\frac{s}{P_\lambda(s)}ds:= \lim_{\rho\to \infty} \int_{-\rho}^{\rho}\frac{s}{P_\lambda(s)}ds.
\end{align}

}
We then have the following result.
\begin{theorem}\label{Thm-n=1} Fix any $k\in \mathbb N$ and consider system \eqref{xy} with $n=1$ and $P_{\lambda}(x_2) =\lambda_0 +   \lambda_1 x_2  - x_2^{2}$. Suppose that \eqref{f2n} and \eqref{P2ncond} are satisfied and that 
for every $x_{\text{in}}\in I_{\text{in}}$ there is a $x_{\text{out}}=\Delta_0(x_{\text{in}})\in I_{\text{out}}$ so that \RH{
\begin{align}\eqlab{entry-exit-n=1}
\text{p.v.} \int_{x_{\text{out}}^b}^{x_{\text{in}}^b}\frac{1}{s\zeta_2(s,0)}ds + \text{p.v.} \int_{-\infty}^{+\infty}\frac{s}{P_\lambda(s)}ds=0,
\end{align} }
with $x_{\text{in}}^b=x_{\text{in}}^b(x_{\text{in}})$ and $x_{\text{out}}^b=x_{\text{out}}^b(x_{\text{out}})$ defined above.
Then the Dulac map $\Delta(\cdot,\epsilon):I_{\text{in}}\rightarrow \KUK{\mathbb R}$ is well-defined for all $\epsilon\in ]0,\epsilon_0[$, with $\epsilon_0>0$ small, and takes the following form
\begin{align*}
 \Delta(x_{\text{in}},\epsilon) =\Delta_0(x_{\text{in}})+\phi(x_{\text{in}},\epsilon,\epsilon \log \epsilon^{-1}),
\end{align*}
where $\phi: I_{\text{in}}\times [0,\epsilon_0)\times [0,\epsilon_0 \log \epsilon_0^{-1})\rightarrow \mathbb R$ is $C^k$-smooth and satisfies $\phi(x_{\text{in}},0,0)=0$ for all $x_{\text{in}}\in I_{\text{in}}$. 
\end{theorem}
We prove Theorem \ref{Thm-n=1} in Section \ref{section-proof-n=1}. \NEW{The smoothness w.r.t. $(\epsilon,\epsilon \log \epsilon^{-1})$ is natural since we will deal with the passage near a line of saddle singularities with positive and negative eigenvalues of equal magnitude (see Lemma \ref{lemma-boundary}). However, it might be possible  (although we do not expect it) that the final transition map is in fact smooth w.r.t. $\epsilon$ (as in \cite{DMS2016}). We have not pursued this in the present work since (a) it is not expected to be important for the cyclicity results of $(I_2^1)$ and $(I_4^1)$ and (b) it does not seem like a trivial task. Notice in particular w.r.t. (b) that logarithms also appear due to resonances associated with a separate blow-up transformation (see e.g. \eqref{normalformhatY} below).}

\begin{remark}\label{remark-Cauchy}
 \RH{From \eqref{CauchyPV}, $\zeta_{2}(0,0)=-1$ and 
\begin{align*}
  \frac{1}{s\zeta_2(s,0)}=\frac{\zeta_2(s,0)+1}{s\zeta_2(s,0)}-\frac{1}{s},
 \end{align*} 
 it follows that 
 $$\text{p.v.} \int_{x_{\text{out}}^b}^{x_{\text{in}}^b}\frac{1}{s\zeta_2(s,0)}ds=\int_{x_{\text{out}}^b}^{x_{\text{in}}^b} \frac{\zeta_2(s,0)+1}{s\zeta_2(s,0)}ds+\log\left(-\frac{x_{\text{out}}^b}{x_{\text{in}}^b}\right).$$

Using \eqref{CauchyPV-second-part} and $P_{\lambda}(s) =\lambda_0 +   \lambda_1 s  - s^{2}$, it is not difficult to see that
    \begin{align*}
      \text{p.v.} \int_{-\infty}^{+\infty}\frac{s}{P_\lambda(s)}ds &=  \int_{-1}^1 \frac{s}{P_\lambda(s)}ds+\left(\int_{-\infty}^{-1}+\int_{1}^{\infty}\right) \frac{P_\lambda(s)+s^2}{sP_\lambda(s)}ds\\
      &=-\frac{\lambda_1 \pi}{\sqrt{-4\lambda_0-\lambda_1^2}}.
    \end{align*}
    If we now plug in these expressions into \eqref{entry-exit-n=1}, we get the following entry-exit formula
 \begin{align}\eqlab{EE-remark}
\int_{x_{\text{out}}^b}^{x_{\text{in}}^b} \frac{\zeta_2(s,0)+1}{s\zeta_2(s,0)}ds+\log\left(-\frac{x_{\text{out}}^b}{x_{\text{in}}^b}\right)=\frac{\lambda_1 \pi}{\sqrt{-4\lambda_0-\lambda_1^2}}.
\end{align}

}
\end{remark}

\RH{
\begin{remark}\label{remark(a)(b)}(a)
  The first (resp. second) integral on the right hand side in \eqref{CauchyPV}  is the slow divergence integral \cite[Chapter 5]{DDR-book-SF} of \eqref{xy}, with $n=1$, associated with the normally repelling (resp. attracting) segment $[x_{\text{out}}^b,-\rho]$ (resp. $[\rho,x_{\text{in}}^b]$) of the curve of singularities $\{y=0\}$. These are integrals of the divergence of the vector field \eqref{xy}, for $\epsilon=0$ and calculated along $\{y=0\}$, where the integration variable is the time variable $\tau$ of the flow of the slow vector field \cite[Chapter 3]{DDR-book-SF}
  $$\frac{dx}{d\tau}=x^2\zeta_2(x,0).$$

  Note that the integral in the classical entry-exit formula \eqref{Delta0} is equal to the slow divergence integral of system \eqref{model0} computed along the segment $[x_{\text{out}},x_{\text{in}}]$.

(b) The Cauchy principal value in \eqref{CauchyPV-second-part}  is related to the divergence integral on the second cylinder (see Section \ref{section-blow-up}), defined in terms of $P_{\lambda}$. For more details, we refer the reader to Section \ref{section-proof-n=1} (see Remark \ref{remark-DI}). 
\end{remark}
}

When $n\ge 2$ in \eqref{xy}, we have the following result.
\begin{theorem}\label{thm-n>1}
 Fix any $k\in \mathbb N$ and consider the system \eqref{xy} with $n\ge 2$. Suppose that \eqref{f2n} and \eqref{P2ncond} are satisfied. Then $\int_{-\infty}^\infty \frac{v}{P_\lambda(v)}dv$ is well-defined and we suppose that:
 \begin{align}
 \int_{-\infty}^\infty \frac{v}{P_\lambda(v)}dv \ne 0.\eqlab{condI}
 \end{align}
 Then the entry-exit problem $I_{\text{in}}\ni x_{\text{in}}\mapsto \Delta(x_{\text{in}},\epsilon)\in I_{\text{out}}$ is not well-defined for all $0<\epsilon\ll 1$. In further details, let $(0,\e^{-1/z_{\text{in/out}}  })$ denote the intersection points of the forward and backward flow of $(x_{\text{in}},\delta)\in \Sigma_{\text{in}}$ and  $(x_{\text{out}},\delta)\in \Sigma_{\text{out
}}$, respectively, with $\{x=0\}$. Then 
\begin{equation}\eqlab{zinzout}
\begin{aligned}
 z_{\text{in}} =z_{\text{in}}(x_{\text{in}},\epsilon)&=\epsilon^{2n-1}\left(\frac{1}{-\int_0^{\infty}\frac{v}{P_\lambda(v)}dv}+\phi_{\text{in}}(\KUK{x_{\text{in}}},\epsilon,\epsilon \log \epsilon^{-1})\right),\\
 z_\text{out} =z_{\text{out}}(x_{\text{out}},\epsilon)&=\epsilon^{2n-1}\left(\frac{1}{\int_{-\infty}^0\frac{v}{P_\lambda(v)}dv}+\phi_{\text{out}}(\KUK{x_{\text{out}}},\epsilon,\epsilon \log \epsilon^{-1})\right),
\end{aligned}
\end{equation}
with each $\phi_\text{in/out} :I_\text{in/out} \times [0,\epsilon_0)\times [0,\epsilon_0 \log \epsilon_0^{-1})\rightarrow \mathbb R$ being $C^k$ smooth and $\phi_\text{in/out} (x,0,0)=0$ for all $x\in I_{\text{in/out}}$. Therefore if \eqref{condI} holds true, then $z_\text{in}\ne z_\text{out}$ for all $x_{\text{in}}\in I_{\text{in}},x_{\text{out}}\in I_{\text{out}}$, $0<\epsilon\ll 1$.
\end{theorem}
We prove Theorem \ref{thm-n>1} in Section \ref{proof-n>2}.
\begin{remark}
In contrast to Theorem \ref{Thm-n=1}, the contraction/expansion towards $y=0$ is dominated by the passage near $x=0$ for $n\ge 2$. Indeed, on either side of $x=0$,  transition maps $y\mapsto y_+$ between different sections $\{x=x_{0}\}$ and $\{x=x_1\}$, \KUK{$x_0x_1>0$}, are to leading order given by $y_+=\exp(\epsilon^{-1} I)$ with 
\begin{align*}
    \KUK{I = \int_{x_0}^{x_1} -\frac{x}{f_\lambda(x,0)}dx}.
\end{align*}
This follows from the theory of slow-divergence integrals, see \cite{DDR-book-SF}.
On the other hand, if we put $x=\epsilon x_2$, then on either side of $x_2=0$,  transition maps $y\mapsto y_+$ between different sections $\{x_2=x_{20}\}$ and $\{x_2=x_{21}\}$, $x_{20}x_{21}>0$, are to leading order given by $\exp(\epsilon^{-2n+1} I_{2})$ with
\begin{align*}
    I_{2}=\int_{x_{20}}^{x_{21}} -\frac{x_2}{P_\lambda(x_2)}dx_2.
\end{align*} 
We obtain this by substituting $x=\epsilon x_2$ into \eqref{xy}; for the purpose of this remark, we have just retained the dominating terms \KUK{(and ignored $g$ in \eqref{xy})}.  Since $-2n+1<-1$ for $n\ge 2$ this illustrates our claim. For $n=1$, the contractions/expansions for $x=\mathcal O(1)$ and $x=\mathcal O(\epsilon)$ are comparable. 
\end{remark}

The sign of the integral \eqref{condI} determines the sign of $z_{\text{in}}-z_{\text{out}}$. In particular, if 
\begin{align*}
     \int_{-\infty}^\infty \frac{v}{P_\lambda(v)}dv<0,
\end{align*}
then 
\begin{align*}
 0<   z_{\text{in}}<z_{\text{out}},
\end{align*}
recall \eqref{zinzout}, \NEW{for all $0<\epsilon\ll 1$, $x_{\text{in}/\text{out}}\in I_{\text{in}/\text{out}}$}. 
This follows from \eqref{zinzout} using $P_\lambda(x_2)<0$. 
Since the problem is planar, a simple corollary of this fact is that the forward orbit of $(x_{\text{in}},\delta)\in \Sigma_{\text{in}}$ for any $x_{\text{in}}\in I_{\text{in}}$ intersects $\{x=x_{\text{out}}\}$ with $y=\mathcal O(\e^{-c\epsilon^{-2n+1}})$, \KUK{$c>0$}, for $\epsilon\to 0$ for any $x_{\text{out}}\in I_{\text{out}}$. 

On the other hand, if 
\begin{align*}
     \int_{-\infty}^\infty \frac{v}{P_\lambda(v)}dv>0,
\end{align*}
then 
\begin{align*}
    z_{\text{in}}>z_{\text{out}}>0,
\end{align*}
and the forward orbit of $(x_{\text{in}},\delta)\in \Sigma_{\text{in}}$ for any $x_{\text{in}}\in I_{\text{in}}$
intersects $\{y=\delta\}$ with $x=o(1)<0$ for $\epsilon\to 0$. 

The remaining case
\begin{align*}
     \int_{-\infty}^\infty \frac{v}{P_\lambda(v)}dv=0,
\end{align*}
is similar to classical canard situation, see \cite{DDR-book-SF}, when we treat $\lambda$ as control parameters. Indeed, we have the following.

\begin{cor}
   {Consider as in Theorem \ref{thm-n>1} any $k\in\mathbb N$, $n\ge 2$ and denote $\Delta(x,\epsilon)$ by $\Delta(x,\epsilon;\lambda)$ to emphasize its dependency on $\lambda=(\lambda_0,\ldots,\lambda_{2n-1})\in \Lambda$. Now, let $l\in \{0,\ldots,2n-1\}$ and $\lambda^0\in \Lambda$ be such that 
   \begin{align}
       \int_{-\infty}^\infty \frac{v}{P_{\lambda^0}(v)}dv = 0\quad \mbox{and}\quad \frac{\partial }{\partial \lambda_l }  \int_{-\infty}^\infty \frac{v}{P_{\lambda}(v)} dv\bigg|_{\lambda=\lambda^0} \ne  0.\eqlab{cond111}
   \end{align}
   Then there is an $(x_\text{in},x_\text{out},\epsilon)$-dependent local embedding defined by $$\widehat \Lambda \ni \widehat \lambda\mapsto \lambda = \overline \lambda(x_\text{in},x_\text{out},\epsilon,\widehat \lambda)\in \mathbb R^{2n},$$ with $\widehat \Lambda\subset \mathbb R^{2n-1}$ being a sufficiently small neighborhood of the origin in $\mathbb R^{2n-1}$, and where $\overline \lambda$ is $C^k$-smooth w.r.t. $$\KUK{(x_\text{in},x_{\text{out}},\epsilon,\epsilon \log \epsilon^{-1},\widehat \lambda)\in I_{\text{in}}\times I_{\text{out}}\times [0,\epsilon_0)\times [0,\epsilon_0 \log \epsilon_0^{-1}) \times \widehat \Lambda,\quad 0<\epsilon_0\ll 1},$$ such that 
   \begin{align*}
       \overline \lambda(x_\text{in},x_\text{out},0,0)=\lambda^0\quad \forall\,x_{\text{in}}\in I_{\text{in}},\,x_{\text{out}}\in I_{\text{out}},
   \end{align*}
   and 
       \begin{align*}
    \Delta(x_{\text{in}},\epsilon;\overline \lambda(x_{\text{in}},x_{out},\epsilon,\widehat \lambda))=x_{\text{out}}\quad \forall\,(x_\text{in},x_{\text{out}},\epsilon,\widehat \lambda)\in I_{\text{in}}\times I_{\text{out}}\times [0,\epsilon_0)\times \widehat \Lambda. 
    \end{align*}
    
    }
\end{cor}
 \begin{proof}
   {We have $x_{\text{out}}=\Delta(x_{\text{in}},\epsilon;\lambda)$ if and only if $z_{\text{in}}(x_{\text{in}},\epsilon)=z_{\text{out}}(x_{\text{out}},\epsilon)$. From \eqref{zinzout}, this reduces to
     \begin{align*}
         \int_{-\infty}^\infty \frac{v}{P_\lambda(v)}dv = o(1).
     \end{align*}
     Here $o(1)$ is a $C^k$-smooth function w.r.t. $x_{\text{in}},x_{\text{out}},\epsilon,\epsilon \log \epsilon^{-1}$ and $\lambda\in \Lambda$, which vanishes for $\epsilon=0$. This follows from the proof of Theorem \ref{thm-n>1}. The result then follows from a simple application of the implicit function theorem. 
     }
 \end{proof}
\begin{remark}
{Let $\lambda^0=(\lambda_0^0,\ldots,\lambda_{2n-1}^0)\in \Lambda$ be so that $\lambda_j^0=0$ for all odd $j\in \{1,3,\ldots,2n-1\}$. Then the first condition in \eqref{cond111} clearly holds since the integrand is an odd function. Moreover, for any odd \KUK{$l\in \{1,3,\ldots,2n-1\}$}, we find that
 \begin{align*}
     \frac{\partial }{\partial \lambda_l }  \int_{-\infty}^\infty \frac{v}{P_{\lambda}(v)} dv\bigg|_{\lambda=\lambda^0} = -\int_{-\infty}^\infty \frac{v^{1+l}}{P_{\lambda^0}(v)^2}dv,
 \end{align*}
which is negative since the integrand is an even function of $v$. }
\end{remark}

\section{Blow-up}\label{section-blow-up}
Following \cite{DMS2016}, we define $z\ge 0$ through
\begin{align}\eqlab{yz}
 y = \begin{cases} \e^{-1/z}& \mbox{for $z>0$},\\
 0 & \mbox{for $z=0$}.
 \end{cases}
\end{align}
Then \eqref{xy} becomes
\begin{equation}\eqlab{xz}
\begin{aligned}
 \dot x &= \epsilon f_\lambda(x,\epsilon) + \e^{-1/z} g(x,\e^{-1/z},\epsilon),\\
 \dot z &= - x z^2,\\
 \dot \epsilon &=0,
\end{aligned}
\end{equation}
after augmenting a trivial equation for $\epsilon$. 
Notice that $(x,0,0)$ defines a line of degenerate singularities, the linearization having only zero eigenvalues. The transformation \eqref{yz} enables the use of blow-up  for the entry-exit problem. In particular, we consider the cylindrical blow-up   transformation
\begin{align}\eqlab{cylbu}
 \rho\ge 0,\,(\bar z,\bar \epsilon)\in \mathbb S^1 \mapsto \begin{cases}
                                                    z &= \rho \bar z,\\
                                                    \epsilon &=\rho \bar \epsilon,
                                                   \end{cases}
\end{align}
leaving $x$ fixed, and use a desingularization corresponding to division of the pull-back vector field by $\rho$. Note that $\bar z\ge 0$ and $\bar\epsilon\ge 0$. We use two separate charts $\bar z=1$ and $\bar \epsilon=1$ with chart-specific coordinates $(\rho_1,\epsilon_1)$ and $(z_2,\rho_2)$ defined by
\begin{align}
 \bar z = 1:\quad \begin{cases}
                   z &=\rho_1,\\
                   \epsilon &=\rho_1\epsilon_1,
                  \end{cases}\nonumber \\
  \bar \epsilon = 1:\quad \begin{cases}
                   z &=\rho_2 z_2,\\
                   \epsilon &=\rho_2.
                  \end{cases}\eqlab{z2here} \end{align}
                  Here, the desingularization is achieved by division of the vector field by $\rho_1$ resp. $\rho_2$. The change of coordinates is well-defined for $z_2>0$ and are given by the expressions
                  \begin{align}\eqlab{cc1}
                   \begin{cases}
                    \rho_1 &= \rho_2 z_2,\\
                    \epsilon_1 &= z_2^{-1}.
                   \end{cases}
                  \end{align}

 \begin{figure}[!ht] 
\begin{center}
{\includegraphics[width=.95\textwidth]{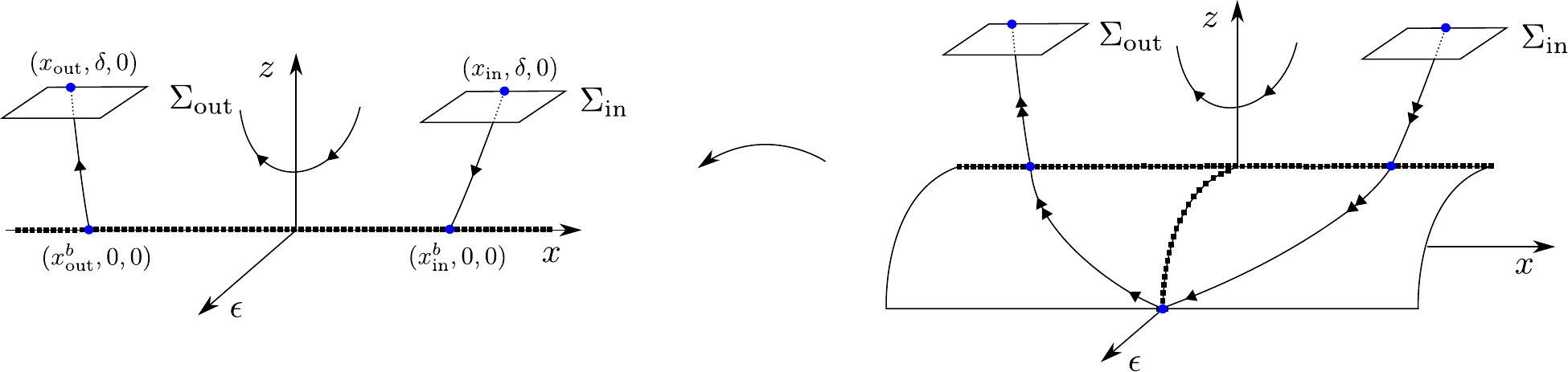}}
\end{center}
 \caption{Cylindrical blowup of the degenerate line $(x,z,\epsilon)=(x,0,0)$ by \eqref{cylbu}.}
\figlab{blowup1}
\end{figure}

                  \subsection{Chart $\bar z=1$}
                  In this chart, we obtain the desingularized vector field defined by
                 \begin{equation}\eqlab{xz1}
\begin{aligned}
 \dot x &= \epsilon_1 f_\lambda(x,\rho_1\epsilon_1) + \rho_1^{-1} \e^{-1/\rho_1} g(x,\e^{-1/\rho_1},\rho_1\epsilon_1),\\
 \dot \rho_1 &= - x \rho_1,\\
  \dot \epsilon_1 &=x \epsilon_1,
\end{aligned}
\end{equation}
which extends smoothly to $\rho_1=0$. In particular, the last term in the equation for $x$ is $C^\infty$ flat w.r.t. $\rho_1=0$, uniformly w.r.t. $x$ and $\epsilon_1$. The point $(x,0,0)$, with $x\ne 0$, is therefore a semi-hyperbolic singularity of \eqref{xz1}, the linearization having eigenvalues $0,-x,x$. Moreover, $\epsilon_1=0$ and $\rho_1=0$ define two invariant sets of \eqref{xz1}. 
Within the latter, we find that
\begin{align*}
 \dot x &= \epsilon_1 f_\lambda(x,0) = \epsilon_1 x^{2n} \zeta_{2n}(x,0),\\
 \dot \epsilon_1 &=x \epsilon_1,
\end{align*}
or as a first order system for $\epsilon_1\ne 0$ 
\begin{align}
 \frac{dx}{d\epsilon_1} &={x^{2n-1} \zeta_{2n}(x,0)},\eqlab{eps1eqn}
\end{align}
which extends smoothly to $\epsilon_1=0$. 
 Let $x_{\text{out},1}=x_{\text{out},1}(x_{\text{in}}^b,\epsilon_1)$, $\epsilon_1>0$, denote the unique solution to \eqref{eps1eqn} with the initial condition $ x_{\text{out},1}(x_{\text{in}}^b,0)=x_{\text{in}}^b$, where $x_{\text{in}}^b>0$ is defined in Section \ref{subsection-statement}. From \eqref{f2n} and \eqref{eps1eqn} it follows that $x_{\text{out},1}$ is a decreasing function in $\epsilon_1$. By the separation of variables, $x_{\text{out},1}$ is implicitly defined by
\begin{align}\eqlab{eps1sol}
 \epsilon_1 = \int_{x_{\text{in}}^b}^{x_{\text{out},1}} \frac{1}{s^{2n-1} \zeta_{2n}(s,0)}ds.
\end{align}

We now define the following sections: 
\begin{align*}
\Sigma_{\text{in},1}\,&:\quad x\in I_{\text{in}},\,\rho_1=\delta_1,\,\epsilon_1 \in [0,\nu), \\
\Sigma_{\text{out},1}\,&:\quad x\in I_{\text{out},1},\,\rho_1 \in [0,\nu),\,\epsilon_1=\delta_1,
\end{align*}
where $I_{\text{in}}$ is defined in Section \ref{subsection-statement}, $I_{\text{out},1}$ is an appropriate interval, $\delta_1=-1/\log\delta$, with $\delta\in (0,1)$ small and fixed in Section \ref{subsection-statement}, and $\nu>0$ is small enough. Notice that $y=\delta$ corresponds to $\rho_1=\delta_1$, due to \eqref{yz}.
\begin{lemma}\label{lemma-boundary}
Fix any $k\in \mathbb N$.
Then the transition map 
\begin{align*}
\Sigma_{\textnormal{in},1}&\rightarrow \Sigma_{\textnormal{out},1},\\
(x_{\text{in}},\delta_1,\epsilon_1)&\mapsto (x_+(x_{\text{in}},\epsilon_1), \epsilon_1,\delta_1),
\end{align*}
given by the forward flow of \eqref{xz1}, is well-defined for all $\epsilon_1\in ]0,\epsilon_0[$ with $\epsilon_0>0$ small enough. In particular, $x_+$ takes the following form:
\begin{align*}
 x_+(x_{\text{in}},\epsilon_1) =x_{\textnormal{out},1}(x_{\text{in}}^b(x_{\text{in}}),\delta_1)+ o(1),
\end{align*} 
 where $o(1)$ is $C^k$-smooth w.r.t. $(x_{\text{in}},\epsilon_1, \epsilon_1 \log \epsilon_1)$ and is zero for all $x_{\text{in}}\in I_{\text{in}}$ when $\epsilon_1=0$. 
\end{lemma}
\begin{proof}
Since we deal with the passage near the line of saddle singularities with positive and negative eigenvalues of equal magnitude, this result follows from \cite{DMS2016}.
\end{proof}

                  \subsection{Chart $\bar \epsilon=1$}\label{subsection-chart-bar-epsilon}
                 In this chart, we obtain the desingularized vector field defined by
                 \begin{equation}\eqlab{xz2}
\begin{aligned}
 \dot x &=  f_\lambda(x,\rho_2) + \rho_2^{-1} \e^{-1/(\rho_2 z_2)} g(x,\e^{-1/(\rho_2z_2)},\rho_2),\\
   \dot z_2 &=-x z_2^2,\\
   \dot \rho_2 &= 0.   
\end{aligned}
\end{equation}
  The invariant behavior on the plane $z_2=0$ is given by 
 \begin{equation}
 \begin{aligned}
 \dot x &=  f_\lambda(x,\rho_2),\nonumber\\
   \dot \rho_2 &= 0. \nonumber  
\end{aligned}
\end{equation}
The solutions are horizontal lines in the $(x,\rho_2)$-plane and $f_\lambda(x,\rho_2)<0$ for $x\in I$, $\lambda\in \Lambda$ and $\rho_2>0$ small enough (see Lemma \ref{lemma-condition}).

Notice that within $\rho_2=0$ we have
\begin{align*}
 -z_2^{-2}\frac{dz_2}{dx} = \frac{1}{x^{2n-1} \zeta_{2n}(x,0)},
\end{align*}
for $z_2\ne 0$.
Here we have used $f_\lambda(x,0)=x^{2n} \zeta_{2n}(x,0)$. 
We can solve this equation for $z_2=z_2(x)$:
\begin{align*}
 z_2(x) = \frac{1}{z_2(x_{0})^{-1} + \int_{x_{0}}^{x} \frac{1}{s^{2n-1}\zeta_{2n}(s,0)}ds },
\end{align*}
using an initial condition at $x=x_{0}$. The specific solution
\begin{align}\eqlab{z2xb}
 z_2(x) = \frac{1}{\int_{x_{\text{in}}^b}^{x}\frac{1}{s^{2n-1} \zeta_{2n}(s,0)}ds},
\end{align}
corresponds to \eqref{eps1sol} (cf. \eqref{cc1}) with $z_2(x)\rightarrow \infty$ as $x\rightarrow (x_{\text{in}}^b)^{-}$. 
Notice also that $z_2(x)$ in \eqref{z2xb} tends to 0 as $x\to 0^+$. \KUK{In fact, one can easily show (using $\zeta_{2n}(0,0)=-1$) that $\lim_{x\to 0^+} z_2(x) \log x^{-1} =1$ for $n=1$ whereas $\lim_{x\to 0^+} z_2(x) x^{-2(n-1)}=2(n-1)$ for $n\ge 2$, see \figref{new}. (This is the first indication that $n=1$ and $n\ge 2$ are different.) }

 \begin{figure}[!ht] 
\begin{center}
{\includegraphics[width=.45\textwidth]{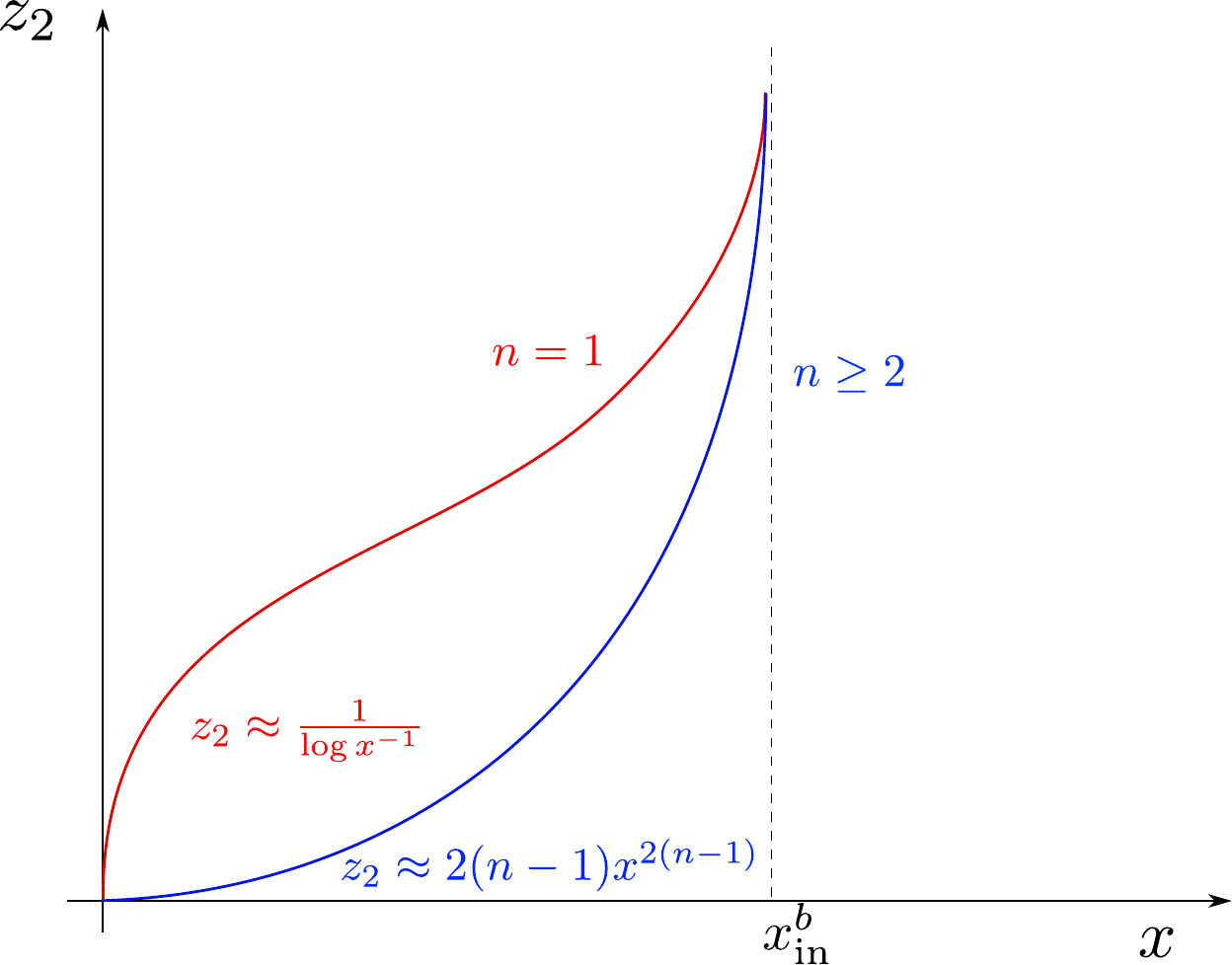}}
\end{center}
 \caption{\NEW{Sketch of \eqref{z2xb} for $n=1$ (red) and $n\ge 2$ (blue).} }
\figlab{new}
\end{figure}

The point $(0,0,0)$ is a degenerate singularity of \eqref{xz2}, the linearization having only zero eigenvalues. We therefore introduce a spherical blow-up for any $n\ge 2$:
\begin{align}\eqlab{bu-spherical}
 r\ge 0,\,(\bar x,\bar z_2,\bar \rho_2) \in \mathbb S^2 \mapsto \begin{cases}
                                                    x &= r\bar x,\\
                                                    z_2 &= r^{2(n-1)} \bar z_2,\\
                                                    \rho_2 &=r \bar \rho_2,
                                                   \end{cases}
\end{align}
and use a desingularization corresponding to division of the right hand side by $r^{2n-1}$, see \figref{blowup2}.

For $n=1$, we use a cylindrical blow-up:
\begin{align}\eqlab{bu2}
 r\ge 0,\,(\bar x,\bar \rho_2) \in \mathbb S^1 \mapsto \begin{cases}
                                                    x &= r\bar x,\\                                                                                                        \rho_2 &=r \bar \rho_2,
                                                   \end{cases}
\end{align}
leaving $z_2$ fixed, and a desingularization corresponding to division of the right hand side by $r$.

 \begin{figure}[!ht] 
\begin{center}
{\includegraphics[width=.95\textwidth]{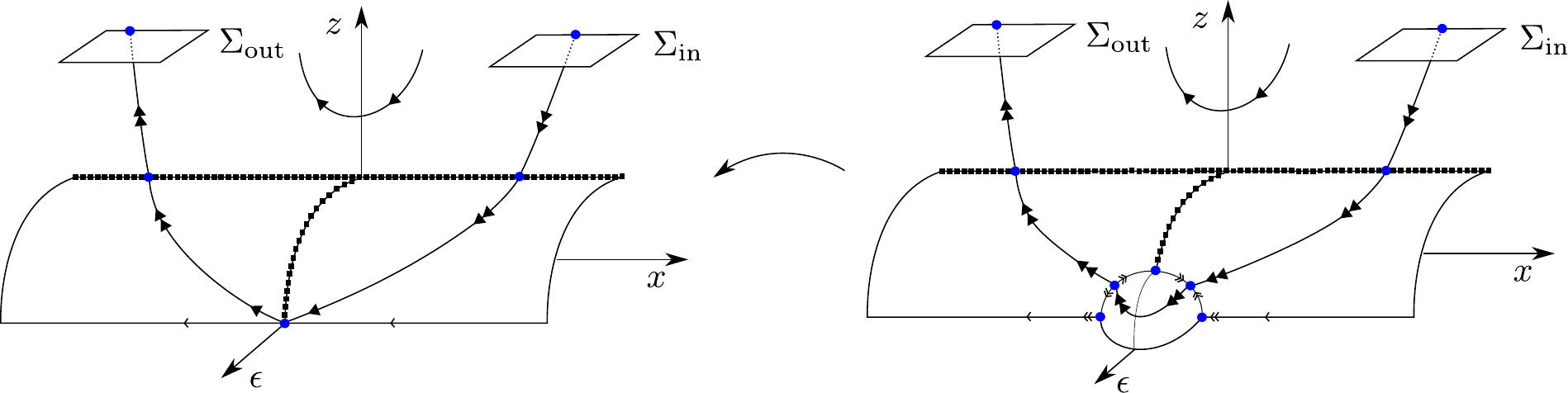}}
\end{center}
 \caption{Illustration of the spherical blow-up \eqref{bu-spherical} for $n\ge 2$. }
\figlab{blowup2}
\end{figure}

We consider two separate charts for any $n\in \mathbb N$: $\bar x=1$ and $\bar \rho_2=1$ with chart-specific coordinates $(r_1,z_{21},\rho_{21})$ and $(x_2,z_{22},r_2)$, respectively, defined by
\begin{align*}
 \bar x = 1 :\quad \begin{cases}
                    x &=r_1,\\
                    z_2 &=r_1^{2(n-1)} z_{21},\\
                    \rho_2 &= r_1 \rho_{21},
                   \end{cases}\\
     \bar \rho_2 = 1 :\quad \begin{cases}
                    x &=r_2 x_2,\\
                    z_2 &=r_2^{2(n-1)} z_{22},\\
                    \rho_2 &= r_2.
                   \end{cases}  \\             
\end{align*}
In each chart, the desingularization is achieved through division of the right hand side by $r_i^{2n-1}$, $i=1,2$. Notice that for $n=1$, $z_2=z_{21}=z_{22}$ is fixed. The change of coordinates between $\bar x=1$ and $\bar \rho_{2}=1$ is well-defined for $x_2>0$ and given by the expressions:
\begin{align}\eqlab{cc12}
 \begin{cases} r_1 &= r_2 x_2,\\
 z_{21} &= z_{22} x_2^{-2(n-1)},\\
 \rho_{21} &= x_{2}^{-1}.
 \end{cases}
\end{align}
We divide the analysis into $n=1$ (Section \ref{section-proof-n=1}) and $n\ge 2$ (Section \ref{proof-n>2}).

\section{Proof of Theorem \ref{Thm-n=1}}\label{section-proof-n=1}
In this section, we prove Theorem \ref{Thm-n=1}. We therefore consider the system \eqref{xz}, with $n=1$ and $P_{\lambda}(x_2) =\lambda_0 +   \lambda_1 x_2  - x_2^{2}$, and assume that \eqref{f2n} and \eqref{P2ncond} are satisfied.

Using \eqref{flambda} and \eqref{xz2}, the desingularization of \eqref{xz} in the chart $\bar \epsilon=1$ yields
\begin{equation}\nonumber
\begin{aligned}
 \dot x &=  \rho_2^{2} P_\lambda(\rho_2^{-1} x) +x^{2} ( \zeta_2(x,\rho_2)+1) + \rho_2^{-1} \e^{-1/(\rho_2 z_2)} g(x,\e^{-1/(\rho_2z_2)},\rho_2),\\
   \dot z_2 &=-x z_2^2,\\
   \dot \rho_2 &= 0.
\end{aligned}
\end{equation}
We apply now the blow-up \eqref{bu2}, working in the charts $\bar x=1$ and $\bar \rho_2=1$. Recall that $z_2$ is not transformed under \eqref{bu2} for $n=1$ and in this situation it is in fact more useful to define $y_2$ by 
\begin{align}
 y_2 = \e^{-1/z_2}.\eqlab{y2z2}
\end{align}

\begin{remark}
Notice that \eqref{y2z2} corresponds to $y_2 = y^{\epsilon}$ (by \eqref{z2here} and \eqref{yz}). Moreover, since $z_2(x)\log x^{-1}\to 1$ as $x\to 0^+$ in \eqref{z2xb} for $n=1$ and $\epsilon=0$, we have $y_2(x) x^{-1}\to 1$ as $x\to 0^+$ for $\epsilon=0$ in terms of $y_2$. 

One of the reasons why we use $y_2$ instead of $z_2$ is that in the chart $\bar x=1$, we then deal with passage near a hyperbolic saddle at $(r_1,y_2,\rho_{21})=(0,0,0)$ (see system \eqref{xz21n1}) and use normal forms from \cite{NForms,ZR02}. Working with the original variable $z_2$, one would have a semi-hyperbolic singularity at $(r_1,z_2,\rho_{21})=(0,0,0)$ with eigenvalues of the linearization $-1,0,1$. We have not found suitable normal forms to deal with this semi-hyperbolic case.
\end{remark}

This gives 
\begin{equation}\eqlab{system-after-exp}
\begin{aligned}
 \dot x &=  \rho_2^{2} P_\lambda(\rho_2^{-1} x) +x^{2} ( \zeta_2(x,\rho_2)+1) + \rho_2^{-1} y_2^{1/\rho_2} g(x,y_2^{1/\rho_2},\rho_2),\\
   \dot y_2 &=-x y_2,\\
   \dot \rho_2 &= 0,
\end{aligned}
\end{equation}
In the chart $\bar \epsilon=1$, we consider $z_2\in [0,\mu]$, with $\mu>0$, and therefore work with $y_2\in [0,\e^{-1/\mu}]\subset [0,1)$. 
\subsection{The chart $\bar x=1$}
In this chart, \RH{we use $x=r_1,\rho_2=r_1\rho_{21}$ and } obtain the desingularized vector-field defined by
\begin{equation}\eqlab{xz21n1}
\begin{aligned}
 \dot r_1 &=  -r_1 F_{21}(r_1,y_{2},\rho_{21}),\\
   \dot y_{2} &=-y_2,\\
   \dot \rho_{21} &= \rho_{21} F_{21}(r_1,z_{21},\rho_{21}),
\end{aligned}
\end{equation}
where
\begin{align*}
 F_{21}(r_1,y_2,\rho_{21}) = -Q_\lambda(\rho_{21}) -  \zeta_2(r_1,r_1\rho_{21})-1  +G_{21}( r_1,r_1^{-\KUK{3}}\rho_{21}^{-1} y_2^{1/(r_1\rho_{21})},\rho_{21}),
\end{align*}
with
\begin{align}\eqlab{Q}
 Q_\lambda(\rho_{21}) := \rho_{21}^2 P_\lambda(\rho_{21}^{-1}),
\end{align}
and
\begin{align}\eqlab{G21}
 G_{21} (r_1,q,\rho_{21}) := \KUK{-}q g(r_1,r_1^{\KUK{3}} \rho_{21} q,r_1\rho_{21}).\nonumber
\end{align}
\RH{Notice that $F_{21}(0,y_2,0)=1$ and  $(r_1,y_2,\rho_{21})=(0,0,0)$ is a hyperbolic saddle for \eqref{xz21n1}.}

In the following we fix $N\in \mathbb N$ large enough. Then by working in a sufficiently small neighborhood of $(r_1,y_2,\rho_{21})=(0,0,0)$, we may assume that $F_{21}\in C^N$. We divide the right hand side of \eqref{xz21n1} by $F_{21}$ to obtain the equivalent system
\begin{equation}\eqlab{xz210}
\begin{aligned}
 \dot r_1 &=  -r_1 ,\\
   \dot y_{2} &=-\frac{y_2}{F_{21}(r_1,y_{2},\rho_{21})},\\
   \dot \rho_{21} &= \rho_{21},
\end{aligned}
\end{equation}
in a small neighborhood of the origin, which we now proceed to normalize. 
 \NEW{Since we are interested in an explicit entry-exit relation, we will need detailed information about the normal form transformations. Our strategy is therefore based upon first considering partial linearizations within} the two invariant planes $r_1=0$ and $\rho_{21}=0$. Within the former, we find
\begin{align*}
 \dot y_2 &= \frac{y_2}{Q_\lambda(\rho_{21})},\\
 \dot \rho_{21}&=\rho_{21}. 
\end{align*}
We see that $(y_2,\rho_{21})=(0,0)$ is a (resonant) hyperbolic saddle, with eigenvalues $-1$ and $1$. There are no resonant terms and the system can be linearized explicitly by \RH{$(Y,\rho_{21})\mapsto (y_2,\rho_{21})$} defined by
\begin{align*}
 y_2 &= \e^{\int_0^{\rho_{21}} \frac{Q_\lambda(s)+1}{s Q_\lambda(s)}ds} Y,
\end{align*}
so \NEW{that $\dot Y=-Y$.}
This follows from a simple calculation. Notice that the integral is well-defined since $Q_\lambda(0)=-1$. 

Now, within $\rho_{21}=0$ we find 
\begin{align*}
 \dot r_1 &=-r_1,\\
 \dot y_2 &= \frac{y_2}{\zeta_2(r_1,0)}.
\end{align*}
We see that $(r_1,y_2)=(0,0)$ is a (resonant) hyperbolic node, with eigenvalues $-1$ and $-1$. However, there are no resonant terms and the system can be linearized explicitly by \RH{$(r_1,Y)\mapsto (r_1,y_2)$} defined by 
\begin{align*}
 y_2 &= \e^{-\int_0^{r_1} \frac{\zeta_2(s,0)+1}{s\zeta_2(s,0)}ds} Y,
\end{align*}
so \NEW{that $\dot Y=-Y$.}
This follows from a simple calculation. Notice that the integral is again well-defined since $\zeta_2(0,0)=-1$.

\RH{Before we combine these transformations, notice that 
\begin{align*}
 -\frac{1}{F_{21}}= \frac{1}{Q_\lambda(\rho_{21})} + \frac{1}{\zeta_2(r_1,0)}+1+
 r_1 \rho_{21} R_0(r_1,\rho_{21})+R_1(r_1,y_2,\rho_{21}),
\end{align*}
where $R_0$ is $C^\infty$-smooth and  $R_1$ is $C^N$-flat w.r.t. $r_1\rho_{21}$.
Then the $C^\infty$ diffeomorphism $(r_1,Y,\rho_{21})\mapsto (r_1,y_2,\rho_{21})$ defined by
\begin{align}\eqlab{theta}
 y_2 &= \Theta (r_1,Y,\rho_{21}):=\e^{\int_0^{\rho_{21}} \frac{Q_\lambda(s)+1}{sQ_\lambda(s)}ds-\int_0^{r_1} \frac{\zeta_2(s,0)+1}{s\zeta_2(s,0)}ds} Y,
\end{align}
 brings \eqref{xz210} into the following system}
\begin{equation}\eqlab{Yeqn}
\begin{aligned}
 \dot r_1 &=-r_1,\\
 \dot Y &= Y\left(-1 + r_1 \rho_{21} R_0(r_1,\rho_{21})+R_1(r_1,\Theta (r_1,Y,\rho_{21}),\rho_{21})\right),\\
 \dot\rho_{21}&=\rho_{21}.
\end{aligned}
\end{equation}
We can then apply normal form theory (see e.g. \cite[Proposition 4.6]{ZR02}) to simplify \eqref{Yeqn} further. 
\begin{lemma}
 Fix any $k\in \mathbb N$. \RH{Then there exists a $C^k$-diffeomorphism $(r_1,\widehat Y,\rho_{21})\mapsto (r_1,Y,\rho_{21})$ defined by
 \begin{align}\eqlab{y2Y}
  Y=\widehat Y(1+\mathcal O(r_1\rho_{21}))
 \end{align}
that brings \eqref{Yeqn}, locally in a sufficiently small neighborhood of $(r_1,Y,\rho_{21})=(0,0,0)$, into the normal form }
\begin{equation}\eqlab{normalformhatY}
\begin{aligned}
 \dot r_1&=-r_1,\\
 \dot{\widehat Y} &=- \widehat Y ( 1+\alpha_0(r_1\rho_{21})),\\
 \dot \rho_{21}&=\rho_{21}.
\end{aligned}
\end{equation}
Here $\alpha_0$ is $C^k$-smooth and satisfies $\alpha_0(0)=0$.
\end{lemma}
\begin{proof}
  The result follows from \cite[Proposition 4.6]{ZR02} (see also \cite{NForms}). \KUK{Notice in particular in comparison with \cite[Proposition 4.6]{ZR02} that $\alpha_{i}=0$ for all $i\in \{1,\ldots,N(k)\}$ (using the notation of \cite{ZR02}) since these resonant terms are absent in \eqref{Yeqn}}.
\end{proof}
We now define the following sections: 
\begin{align*}
\Sigma_{\text{in},21}\,&:\quad y_{2}\in [0,\tilde\mu],\,r_1=\delta,\,\rho_{21} \in [0,\nu), \\
\Sigma_{\text{out},21}\,&:\quad y_2\in [0,\e^{-1/\mu}], \,r_1\in [0,\nu),\,\rho_{21}=\delta,
\end{align*}
\RH{where $\tilde\mu,\nu>0$ are small enough and $\delta>0$ is introduced in Section \ref{subsection-statement} (we can shrink $\delta$ if necessary). We have} \KUK{the following}.

\begin{lemma}\label{lemma-normal-form}
Fix any $k\in \mathbb N$ and $\delta>0$ small enough.
Then the transition map 
\begin{align*}
\Sigma_{\textnormal{in},21}&\rightarrow \Sigma_{\textnormal{out},21},\\
(\delta,y_{2},\rho_{21})&\mapsto (\rho_{21},y_{2+}(\KUK{y_2},\rho_{21}), \delta),
\end{align*}
given by the forward flow of \eqref{xz210}, is well-defined for all $\rho_{21}\in [0,\rho_{210})$ with $\rho_{210}>0$ small enough.
In particular, $y_{2+}$ takes the following from:
\begin{align*}
 y_{2+}(y_{2},\rho_{21}) = \frac{\rho_{21}}{\delta} \e^{\int_0^{\delta} \frac{Q_\lambda(s)+1}{sQ_\lambda(s)}ds+\int_0^{\delta} \frac{\zeta_2(s,0)+1}{s\zeta_2(s,0)}ds}   y_2(1+\KUK{o(1)})
\end{align*}
where $\KUK{o(1)}$ is $C^k$-smooth w.r.t. $(y_2,\rho_{21},\rho_{21}\log \rho_{21}^{-1})$ and zero for all $y_2\in [0,\tilde\mu]$ when $\rho_{21}=0$.
\end{lemma}
\begin{proof}
 We simply integrate the normal form \eqref{normalformhatY} and use \RH{the changes of coordinates \eqref{theta} and} \eqref{y2Y}.
\end{proof}

\begin{remark}\remlab{1map}
 We are now in a position to describe the mapping $(x_{\text{in}},\delta)\mapsto (r_1,y_{2+}(x_{\text{in}},\epsilon),\delta)$ (with $r_1=\epsilon \delta^{-1}$) from $\Sigma_{\text{in}}$ to $\Sigma_{\text{out},21}$, \RH{with $\Sigma_{\text{in}}$ defined in Section \ref{subsection-statement}.} Indeed, upon using \RH{Lemma \ref{lemma-boundary}, \eqref{z2xb}, the change of coordinates $y_2 = \e^{-1/z_2}$ and Lemma \ref{lemma-normal-form} } we find that 
 \begin{align*}
  y_{2+} = \frac{\rho_{21}}{\delta} \e^{\int_0^{\delta} \frac{Q_\lambda(s) +1}{sQ_\lambda(s)}ds+\int_0^{\delta} \frac{\zeta_2(s,0)+1}{s\zeta_2(s,0)}ds +\int_{\delta}^{x_{\text{in}}^b} \frac{1}{s\zeta_2(s,0)}ds } (1+o(1)),
 \end{align*}
\RH{ where $o(1)$ is $C^k$-smooth w.r.t. $(x_{\text{in}},\epsilon, \epsilon \log \epsilon^{-1})$ and is identically zero when $\epsilon=0$.}
 We can simplify this further by writing 
 \begin{align*}
  \frac{1}{s\zeta_2(s,0)}=\frac{\zeta_2(s,0)+1}{s\zeta_2(s,0)}-\frac{1}{s},
 \end{align*}
 and by noticing that $\rho_{21}$ here is the value of $\rho_{21}$ at the section $\Sigma_{\text{in},21}$. Therefore $\rho_{21} = \epsilon \delta^{-1}$ since $r_1=\delta$ there.
This gives
 \begin{align*}
  y_{2+} &= \frac{\epsilon}{\delta^2}\e^{\int_0^{\delta} \frac{Q_\lambda(s) +1}{sQ_\lambda(s)}ds} \e^{\int_0^{x_{\text{in}}^b} \frac{\zeta_2(s,0)+1}{s\zeta_2(s,0)}ds -\int_{\delta}^{x_{\text{in}}^b} s^{-1} ds } (1+o(1))\\
  &=\frac{\epsilon}{\delta^2}\e^{\int_0^{\delta} \frac{Q_\lambda(s) +1}{sQ_\lambda(s)}ds}\e^{\int_0^{x_{\text{in}}^b} \frac{\zeta_2(s,0)+1}{s\zeta_2(s,0)}ds +\log \frac{\delta}{x_{\text{in}}^b}} (1+o(1))\\
  &=\frac{\epsilon}{\delta}\e^{\int_0^{\delta} \frac{Q_\lambda(s) +1}{sQ_\lambda(s)}ds}\e^{\int_0^{x_{\text{in}}^b} \frac{\zeta_2(s,0)+1}{s\zeta_2(s,0)}ds -\log x_{\text{in}}^b} (1+o(1)),
 \end{align*}
 \RH{ where $o(1)$ has the same property as above.}
\end{remark}

\subsection{Chart $\bar \rho_2=1$} 
Consider again the system \eqref{system-after-exp}. 
In this chart, we \RH{ use $x=r_2x_2,\rho_2=r_2$ (see Section \ref{subsection-chart-bar-epsilon}) and } obtain the desingularized vector-field defined by
\begin{equation}\eqlab{x2y2}
\begin{aligned}
 \dot x_2 &=  P_\lambda(x_2) +x_2^{2} ( \zeta_2(r_2x_2,r_2)+1) + \RH{r_2^{-3} } y_2^{1/r_2} g(r_2x_2,y_2^{1/r_2},r_2),\\
   \dot y_2 &=-x_2 y_2,
   \end{aligned}
\end{equation}
and $\dot r_2=0$. We consider $x_2\in [-\delta^{-1},\delta^{-1}]$, $y_2\in [0,\e^{-1/\mu}]\subset [0,1)$ and $0\le r_2\ll 1$. On this compact set, \eqref{x2y2} is regular \RH{(we use \eqref{P2ncond} and $\zeta_2(0,0)=-1$) } with $y_2=0$ being an invariant set. Within $r_2=0$, we therefore find
\begin{align*}
 \frac{dy_2}{dx_2} = -\frac{x_2y_2}{P_\lambda(x_2)},
\end{align*}
with the solution:
\begin{align}\eqlab{y2x2sol}
 y_2(x_2) = \e^{-\int_{x_{20}}^{x_2} \frac{s}{P_\lambda(s)}ds} y_2(x_{20}),
\end{align}
using an initial condition at $x_2=x_{20}$. 

We now define the following transverse sections: $\Sigma_{\text{in},22}\,:\,x_2=\delta^{-1},y_2\in [0,\e^{-1/\mu}]$ and $\Sigma_{\text{final},22}\,:\,x_2=0,y_2\in [0,\e^{-1/\mu}]$ for  all $0\le r_2\ll 1$. We then have the following 
\begin{lemma}\lemmalab{2map}
 The transition map $\Sigma_{\text{in},22}\rightarrow \Sigma_{\text{final},22}$, $(\delta^{-1},y_2)\mapsto (0,y_{2+}(y_2,r_2))$ \NEW{associated with \eqref{x2y2}} is well-defined for all $0\le r_2\ll 1$. In particular,
 \begin{align*}
  y_{2+}(y_2,r_2) = \e^{\int_{0}^{\delta^{-1}}\frac{s}{P_\lambda(s)}ds} y_2(1+\mathcal O(r_2)),
 \end{align*}
where $\mathcal O(r_2)$ is $C^\infty$-smooth w.r.t. \NEW{$(y_2,r_2)$} and \NEW{equals} zero for $r_2=0$. 
\end{lemma}
\begin{proof}
 We use \eqref{y2x2sol}, regular perturbation theory and the invariance of $y_2=0$.
\end{proof}

\RH{Notice that $x_2=\delta^{-1}$ corresponds to $\rho_{21}=\delta$ (see \eqref{cc12}). Finally, we get} \KUK{the following}.
\begin{proposition}
 The transition map $(x_{\text{in}},\delta)\mapsto (0,y_{2+}(x_{\text{in}},\epsilon))$ from the original section $\Sigma_{\text{in}}\,:x\in I_{\text{in}},\,y=\delta$ to the section $\Sigma_{\text{final},22}$ is well-defined for all $0< \epsilon\ll 1$. In particular,
 \begin{equation}\eqlab{y2p}
 \begin{aligned}
  y_{2+} = \epsilon \exp\bigg(&{\int_0^1 \frac{s}{P_\lambda(s)}ds+\int_{1}^{\infty}\frac{P_\lambda(s)+s^2}{sP_\lambda(s)}ds}+\int_0^{x_{\text{in}}^b} \frac{\zeta_2(s,0)+1}{s\zeta_2(s,0)}ds \\
  &-\log x_{\text{in}}^b+\phi_{\text{in}}(x_{\text{in}},\epsilon,\epsilon \log \epsilon^{-1}) \bigg),
 \end{aligned}
 \end{equation}
with $\phi_{\text{in}}:I_{\text{in}}\times [0,\epsilon_0)\times [0,\epsilon_0 \log \epsilon_{0}^{-1})\rightarrow \mathbb R$ being $C^k$-smooth and satisfying \RH{$\phi_{\text{in}}(x_{\text{in}},0,0)=0$.} 
\end{proposition}
\begin{proof}
 By combining \lemmaref{2map} and \remref{1map} we find that 
 \begin{align*}
  y_{2+} =\frac{\epsilon}{\delta}  \e^{\int_{0}^{\delta^{-1}}\frac{s}{P_\lambda(s)}ds} \e^{\int_0^{\delta} \frac{Q_\lambda(s) +1}{sQ_\lambda(s)}ds}\e^{\int_0^{x_{\text{in}}^b} \frac{\zeta_2(s,0)+1}{s\zeta_2(s,0)}ds-\log x_{\text{in}}^b}  (1+\KUK{o(1)}),
 \end{align*}
 \RH{where $o(1)$ is $C^k$-smooth w.r.t. $(x_{\text{in}},\epsilon, \epsilon \log \epsilon^{-1})$ and equal to zero when $\epsilon=0$.}
Recall that $r_2=\rho_2=\epsilon$. We now use \eqref{Q} and rewrite
\begin{align*}
\int_0^{\delta} \frac{Q_\lambda(s) +1}{sQ_\lambda(s)}ds = \int_{\delta^{-1}}^{\infty} \frac{v^{-1} P_\lambda(v)+v}{P_\lambda(v)}dv,
\end{align*}
upon using the substitution $s=v^{-1}$. 
In combination with 
\begin{align*}
 \frac{s}{P_\lambda(s) } = \frac{s^{-1} P_\lambda(s)+s}{P_\lambda(s)}-s^{-1},
\end{align*}this leads to 
\begin{align*}
  y_{2+} &= \frac{\epsilon}{\delta} \e^{\int_{0}^{\delta^{-1}}\frac{s}{P_\lambda(s)}ds+\int_{\delta^{-1}}^{\infty}\frac{s^{-1} P_\lambda(s)+s}{P_\lambda(s)} ds}\e^{\int_0^{x_{\text{in}}^b} \frac{\zeta_2(s,0)+1}{s\zeta_2(s,0)}ds-\log x_{\text{in}}^b}  (1+\KUK{o(1)}) \\
  &=\epsilon \e^{\int_0^1 \frac{s}{P_\lambda(s)}ds+\int_{1}^{\infty}\frac{P_\lambda(s)+s^2}{sP_\lambda(s)}ds}\e^{\int_0^{x_{\text{in}}^b} \frac{\zeta_2(s,0)+1}{s\zeta_2(s,0)}ds-\log x_{\text{in}}^b}  (1+\KUK{o(1)}).
 \end{align*}
 Each of the integrals are convergent since $P_\lambda(s)+s^2 = \lambda_0 + s\lambda_1$, cf. \eqref{P2n}. By writing $1+\KUK{o(1)}= \e^{\phi_{\text{in}}(x_{\text{in}},\epsilon,\epsilon \log \epsilon^{-1})}$ the result follows. 
\end{proof}

We now consider the map $(x_{\text{out}},\delta)\mapsto (0,y_{2-})$ from $\Sigma_{\text{out}}$ to $\Sigma_{\text{final},22}$ defined by the backward flow. For this we replace $x$ by $-x$ and $t$ by $-t$. This gives \eqref{xy} with $P_\lambda(x)$ and \RH{$\zeta_{2}(x,\epsilon)$} in \eqref{flambda} replaced by $P_\lambda(-x)$ and \RH{$\zeta_{2}(-x,\epsilon)$}, respectively. We then obtain the following expression for $y_{2-}$ by using \eqref{y2p} :
\begin{align*}
y_{2-} &= \epsilon \e^{\int_0^1 \frac{s}{P_\lambda(-s)}ds+\int_{1}^{\infty}\frac{P_\lambda(-s)+s^2}{sP_\lambda(-s)}ds}\e^{\int_0^{-x_{\text{out}}^b} \frac{\zeta_2(-s,0)+1}{s\zeta_2(-s,0)}ds-\log (-x_{\text{out}}^b)} \e^{\phi_{\text{out}}(x_{\text{out}},\epsilon,\epsilon \log \epsilon^{-1})}\\
&=\epsilon \e^{\int_0^{-1} \frac{s}{P_\lambda(s)}ds+\int_{-1}^{-\infty}\frac{P_\lambda(s)+s^2}{sP_\lambda(s)}ds}\e^{\int_0^{x_{\text{out}}^b} \frac{\zeta_2(s,0)+1}{s\zeta_2(s,0)}ds-\log (-x_{\text{out}}^b)}  \e^{\phi_{\text{out}}(x_{\text{out}},\epsilon,\epsilon \log \epsilon^{-1})},
\end{align*}
for some new $\phi_{\text{out}}$ with \RH{$\phi_{\text{out}}(x_{\text{out}},0,0)=0$}, upon using the substitution $s=-v$ in the second equality. Here $x_{\text{out}}^b=x_{\text{out}}^b(x_{\text{out}})<0$ and \RH{$x_{\text{out}}\in I_{\text{out}}$}. 

To solve the entry-exit problem we consider $y_{2+}=y_{2-}$ as an equation for $(x_{\text{in}},x_{\text{out}},\epsilon)$. This gives the following equation
\begin{align*}
 &{\int_0^1 \frac{s}{P_\lambda(s)}ds+\int_{1}^{\infty}\frac{P_\lambda(s)+s^2}{sP_\lambda(s)}ds}\\
 &\qquad \qquad \qquad +\int_0^{x_{\text{in}}^b} \frac{\zeta_2(s,0)+1}{s\zeta_2(s,0)}ds-\log x_{\text{in}}^b +\phi_{\text{in}}(x_{\text{in}},\epsilon,\epsilon \log \epsilon^{-1})  \\
 &={\int_0^{-1} \frac{s}{P_\lambda(s)}ds+\int_{-1}^{-\infty}\frac{P_\lambda(s)+s^2}{sP_\lambda(s)}ds}\\
 & \qquad \qquad \qquad +\int_0^{x_{\text{out}}^b} \frac{\zeta_2(s,0)+1}{s\zeta_2(s,0)}ds -\log (-x_{\text{out}}^b)+\phi_{\text{out}}(x_{\text{out}},\epsilon,\epsilon \log \epsilon^{-1}),
\end{align*}
or simply
\begin{equation}\eqlab{eqn0}
\begin{aligned}
 \int_{-1}^1 \frac{s}{P_\lambda(s)}ds & +\KUK{\left(\int_{-\infty}^{-1} +\int_{1}^{\infty}\right)\frac{P_\lambda(s)+s^2}{sP_\lambda(s)}ds} \\
 &+\int_{x_{\text{out}}^b}^{x_{\text{in}}^b} \frac{\zeta_2(s,0)+1}{s\zeta_2(s,0)}ds +\log\left(-\frac{x_{\text{out}}^b}{x_{\text{in}}^b}\right)= \phi(x_{\text{in}},x_{\text{out}},\epsilon,\epsilon\log\epsilon^{-1}),
\end{aligned}
\end{equation}
setting $\phi(x_{\text{in}},x_{\text{out}},\epsilon,\epsilon \log \epsilon^{-1}):=\phi_{\text{out}}(x_{\text{out}},\epsilon,\epsilon \log \epsilon^{-1})-\phi_{\text{in}}(x_{\text{in}},\epsilon,\epsilon \log \epsilon^{-1})$. 

Using Remark \ref{remark-Cauchy}, it is clear that when $\epsilon\rightarrow 0$
\eqref{eqn0} reduces to the entry-exit formula \eqref{entry-exit-n=1} (which we write again)
$$
\text{p.v.} \int_{x_{\text{out}}^b}^{x_{\text{in}}^b}\frac{1}{s\zeta_2(s,0)}ds + \text{p.v.} \int_{-\infty}^{+\infty}\frac{s}{P_\lambda(s)}ds=0.
$$
Assume that for every $x_{\text{in}}\in I_{\text{in}}$ there exists $x_{\text{out}}\in I_{\text{out}}$ such that \eqref{entry-exit-n=1} is satisfied. From $x_{\text{out}}^b(x_{\text{out}})<0$, $(x_{\text{out}}^b)'(x_{\text{out}})>0$ and the assumption \eqref{f2n} for $n=1$ it follows that the partial derivative of the left-hand side in \eqref{entry-exit-n=1} w.r.t. $x_{\text{out}}$ is negative. Now, the implicit function theorem and \eqref{eqn0} imply Theorem \ref{Thm-n=1}.

\begin{remark} 
    \label{remark-DI} Let us explain the meaning of the Cauchy principal value defined by \eqref{CauchyPV-second-part}.

    When $r_2=0$, the system \eqref{x2y2} becomes 
    \begin{equation}\eqlab{x2y2r_2equal0}
\begin{aligned}
 \dot x_2 &=  P_\lambda(x_2),\\
   \dot y_2 &=-x_2 y_2.
   \end{aligned}
\end{equation}
Recall that $P_\lambda$ is negative (see \eqref{P2ncond}). The divergence integral of \eqref{x2y2r_2equal0} along the regular orbit $y_2=0$ between $x_2=\rho$ and $x_2=-\rho$, with $\rho >0$, is given by 

\begin{align}
    \eqlab{DIy_2=0}
    \int_{\rho}^{-\rho}\frac{P_\lambda'(x_2)-x_2}{P_\lambda(x_2)}dx_2=\log\left(\frac{P_\lambda(-\rho)}{P_\lambda(\rho)}\right)+\int_{-\rho}^{\rho}\frac{x_2}{P_\lambda(x_2)}dx_2.\nonumber
\end{align}
Notice that the divergence of the vector field \eqref{x2y2r_2equal0} is $P_\lambda'(x_2)-x_2$ and $dt=\frac{dx_2}{P_\lambda(x_2)}$.
Since $P_{\lambda}$ is a quadratic polynomial, the logarithmic term tends to $0$ as $\rho\to\infty$. We conclude that
$$\lim_{\rho\to\infty}\int_{\rho}^{-\rho}\frac{P_\lambda'(x_2)-x_2}{P_\lambda(x_2)}dx_2=\text{p.v.} \int_{-\infty}^{+\infty}\frac{x_2}{P_\lambda(x_2)}dx_2.$$
\end{remark}
%
%


%

%



\section{Proof of Theorem \ref{thm-n>1}}\label{proof-n>2}
In this section, we proof Theorem \ref{thm-n>1}. \RH{We consider the system \eqref{xz}, with $n\ge2$, and assume that \eqref{f2n} and \eqref{P2ncond} are satisfied.}

In the chart $\bar \epsilon=1$, we obtain \eqref{xz2} repeated here for convenience \RH{(see also \eqref{flambda}):}
\begin{equation}\eqlab{equation-repeated}
\begin{aligned}
 \dot x &=  \rho_2^{2n} P_\lambda(\rho_2^{-1} x) +x^{2n} ( \zeta_{2n}(x,\rho_2)+1) + \rho_2^{-1} \e^{-1/(\rho_2 z_2)} g(x,\e^{-1/(\rho_2z_2)},\rho_2),\\
   \dot z_2 &=-x z_2^2,\\
   \dot \rho_2 &= 0,
\end{aligned}
\end{equation}
 and apply the blow-up \eqref{bu-spherical}, working in the charts $\bar x=1$ and $\bar \rho_2=1$.  
\subsection{Chart $\bar x=1$}
In this chart, we \RH{use $x=r_1,z_2=r_1^{2(n-1)}z_{21},\rho_2=r_1\rho_{21}$ and} obtain the desingularized vector-field defined by
\begin{equation}\eqlab{xz21}
\begin{aligned}
 \dot r_1 &=  -r_1 F_{21}(r_1,z_{21},\rho_{21}),\\
   \dot z_{21} &=z_{21} F_{21}(r_1,z_{21},\rho_{21}) \left(2(n-1)-\frac{z_{21}}{F_{21}(r_1,z_{21},\rho_{21})}\right),\\
   \dot \rho_{21} &= \rho_{21} F_{21}(r_1,z_{21},\rho_{21}),  \nonumber
\end{aligned}
\end{equation}
after division of the right hand side by $r_1^{2n-1}$,
where 
\begin{align*}
 F_{21}(r_1,z_{21},\rho_{21}) = -Q_\lambda(\rho_{21}) -  \zeta_{2n}(r_1,r_1\rho_{21})-1  +r_1^{2n-3}\rho_{21} z_{21}^2  G_{21}(r_1,r_1^{2n-1} \rho_{21}z_{21},r_1\rho_{21}),
\end{align*}
with 
\begin{align*}Q_\lambda(\rho_{21}):&=\rho_{21}^{2n} P_\lambda(\rho_{21}^{-1}) = \rho_{21}^{2n} \lambda_0+\RH{\cdots +} \rho_{21} \lambda_{2n-1}-1
\end{align*}
and 
\begin{align*}
 G_{21}(x,z, \epsilon): = -z^{-2} \e^{-1/z} g(x,\e^{-1/z},\epsilon).
\end{align*}
In particular, $G_{21}$ is $C^\infty$-flat w.r.t. the second argument. 

Since \RH{$Q_\lambda(0)=-1$ and} $\zeta_{2n}(0,0)=-1$, it follows that $\RH{F_{21}}(0,z_{21},0) = 1$ for all $z_{21}$.  We therefore divide the right hand side by $F_{21}$ to obtain the equivalent system:
\begin{equation}\eqlab{xz21normal}
\begin{aligned}
 \dot r_1 &=  -r_1,\\
   \dot z_{21} &=z_{21}\left(2(n-1)-\frac{z_{21}}{F_{21}(r_1,z_{21},\rho_{21})}\right) ,\\
   \dot \rho_{21} &= \rho_{21},
\end{aligned}
\end{equation}
for $z_{21}\ge 0 $ in a large compact set and for all $(r_1,\rho_{21})$ sufficiently small. System \eqref{xz21normal} has a hyperbolic saddle at $(r_1,z_{21},\rho_{21})=(0,2(n-1),0)$, where the linearization has eigenvalues $-1,-2(n-1),1$\RH{, and  a hyperbolic saddle at $(r_1,z_{21},\rho_{21})=(0,0,0)$, where the eigenvalues of the linear part are given by $-1,2(n-1),1$. We refer to \figref{blowup2}.}

\RH{
\begin{remark} Using the change of coordinates $x=r_1,z_2=r_1^{2(n-1)}z_{21}$, it can be easily seen that $z_2(x)$ in \eqref{z2xb} tends to the hyperbolic saddle at $(r_1,z_{21},\rho_{21})=(0,2(n-1),0)$ as $x=r_1\to 0^+$, recall the discussion below \eqref{z2xb}.
    Notice also that $I_{\text{in}}\subset (0,\infty)$ is kept in a compact interval. This implies that the passage near the hyperbolic saddle at $(r_1,z_{21},\rho_{21})=(0,0,0)$ and the passage near the end point of the line of singularities $x=0$ of \eqref{equation-repeated} with $\rho_2=0$ on the blow up locus, visible in the chart $\bar z_2=1$ of \eqref{bu-spherical}, are not relevant for the proof of Theorem \ref{thm-n>1}. We refer to \figref{blowup2}.
\end{remark}
}

We now seek to normalize \eqref{xz21normal}. We focus on the passage near the hyperbolic saddle at $(r_1,z_{21},\rho_{21})=(0,2(n-1),0)$. \NEW{Our strategy will again be based upon partial linearizations within the two invariant planes $r_1=0$ and $\rho_{21}=0$.}

\NEW{First, we use that
\begin{align}\eqlab{decompo-n>1}
 -\frac{1}{F_{21}}&= \frac{1}{Q_\lambda(\rho_{21})}+ \frac{1}{\zeta_{2n}(r_1,0)} +1 \nonumber \\
 &- r_1 \rho_{21}  R_{0}(r_1,\rho_{21})-R_{1}(r_1,z_{21} ,r_1^{2n-1} \rho_{21}z_{21},\rho_{21}),
\end{align}
with $R_{0}$ and $R_{1}$ $C^\infty$-smooth. This follows from a simple calculation combined with a Taylor expansion (see also Section \ref{section-proof-n=1}).} 
In particular, $R_{1}$ is $C^\infty$-flat w.r.t. its third argument. 
\NEW{We then consider a change of coordinates $(r_1,Z,\rho_{21})\mapsto (r_1,z_{21},\rho_{21})$ defined by}
\begin{align*}
 z_{21} \RH{=\Theta (r_1,Z,\rho_{21}):}= \frac{1}{Z+L(r_1,\rho_{21})},\quad L(0,0)=\frac{1}{2(n-1)},
\end{align*}
with $L\in C^1$. 
Notice that $(r_1,z_{21},\rho_{21})=(0,2(n-1),0)$ corresponds to $(r_1,Z,\rho_{21})=(0,0,0)$. 
\RH{Then the $z_{21}$-component of \eqref{xz21normal} changes into }
\begin{align*}
 \dot Z &= - 2(n-1) Z +R_1(r_1,\Theta (r_1,Z,\rho_{21}),r_1^{\RH{2n-1}}\rho_{21}\Theta (r_1,Z,\rho_{21}),\rho_{21})\\
 &+\left\{-\dot L-2 (n-1) L -\frac{1}{Q_\lambda(\rho_{21})}- \frac{1}{\zeta_{2n}(r_1,0)} -1 + r_1 \rho_{21}  R_{0}(r_1,\rho_{21}) \right\}, 
\end{align*}
where $\dot L = -\frac{\partial L}{\partial r_1} r_1+\RH{\frac{\partial L}{\partial \rho_{21}} }\rho_{21}$. 
We will select $L$ so that the curly bracket is zero and therefore look for an invariant manifold $L\in C^1$,  $L(0,0)=\frac{1}{2(n-1)}$, of the following first order system
\begin{equation}\eqlab{r1Lrho21}
\begin{aligned}
 \dot r_1 &=-r_1,\\
 \dot  L &= -2(n-1)L-\frac{1}{Q_\lambda(\rho_{21})}- \frac{1}{\zeta_{2n}(r_1,0)} -1 + r_1 \rho_{21}  R_{0}(r_1,\rho_{21}),\\
 \dot \rho_{21}&=\rho_{21}.
\end{aligned}
\end{equation}
%

We first consider the invariant subspaces $r_1=0$ and $\rho_{21}=0$. For $r_1=0$, we set $J(\rho_{21})=L(0,\rho_{21})$ and find
\begin{equation}\eqlab{xz21normalr10}
\begin{aligned}
 \dot  J &= -2(n-1)J-\frac{1}{Q_\lambda(\rho_{21})},\\
 \dot \rho_{21}&=\rho_{21}.
\end{aligned}
\end{equation}
Since $Q_\lambda(0)=-1$, $(J,\rho_{21})=(\frac{1}{2(n-1)},0)$ is a hyperbolic saddle \RH{of \eqref{xz21normalr10}}, the linearization having eigenvalues $-2(n-1),1$, and there exists a unique local unstable 
manifold given as the $C^\infty$ graph:
\begin{align}\eqlab{J-def}
J(\rho_{21})=-\rho_{21}^{-2(n-1)} \int_0^{\rho_{21}}\frac{v^{2n-3}}{Q_\lambda(v)}dv= -\int_0^1 \frac{v^{2n-3}}{Q_\lambda(\rho_{21} v)}dv ,\quad \rho_{21}\ge 0.
\end{align}
This is a simple calculation. Notice that $2n-3\ge 1$ for $n\ge 2$ and $J(0)=\frac{1}{2(n-1)}$.

%
%
%

Next, we consider $\rho_{21}=0$ and set $K(r_1)=L(r_1,0)$. Then we find that
\begin{equation}\eqlab{xz21normalrho210}
\begin{aligned}
\dot r_1 &= -r_1,\\
\dot  K &= -2(n-1)K-\frac{1}{\zeta_{2n}(r_1,0)},\\
\end{aligned}
\end{equation}
Since $\zeta_{2n}(0,0)=-1$, we conclude that $(r_1,K)=(0,\frac{1}{2(n-1)})$ is a hyperbolic stable node \RH{of \eqref{xz21normalrho210}} with the linearization having eigenvalues $-1$ and $-2(n-1)$. Therefore there is no smooth invariant manifold solution $K=K(r_1)$ in general, but we will fix $\delta>0$ and consider
\begin{align}
 K(r_1)=r_1^{2(n-1)}\int_\delta^{r_1}\frac{v^{1-2n}}{\zeta_{2n}(v,0)}dv.\eqlab{Kexpr}
\end{align}
It is a simple calculation to show that this defines an invariant manifold solution of \eqref{xz21normalrho210}. Moreover, we have the following:
\begin{lemma}
 \RH{The function} $K$ given by \eqref{Kexpr} is $C^\infty$-smooth w.r.t. $(r_1,r_1^{2(n-1)} \log r_1^{-1} )$ and equals $\frac{1}{2(n-1)}$ for $r_1=0$. 

\end{lemma}
\begin{proof}
We expand the smooth function $\zeta_{2n}(r_1,0)^{-1}$ as follows
\begin{align*}
 \zeta_{2n}(r_1,0)^{-1} =: -1+\sum_{k=1}^{2n-3} a_k r_1^k  +{a_{2(n-1)}} r_1^{2(n-1)} + r_1^{2{n}-1} E(r_1),
\end{align*}
with $E\in C^\infty$.
Then we have
\begin{align*}
 K(r_1) &= \frac{1}{2(n-1)}-\frac{r_1^{2(n-1)} \delta^{2(1-n)}}{2(n-1)}+ \sum_{k=1}^{2n-3} a_k \frac{r_1^{k}-r_1^{2(n-1)} \delta^{2(1-n)+k}}{\RH{2(1-n)+k}} \\
 &+{a_{2(n-1)}} r_1^{2(n-1)} \log \frac{r_1}{\delta}+ \RH{r_1^{2(n-1)}}\int_{\delta}^{r_1} E(v)dv.
\end{align*}
This completes the proof. 
\end{proof}

\begin{remark}
    By slight abuse of notation, we will write $K$ given by \eqref{Kexpr} as $$K(r_1,r_1^{2(n-1)} \log r_1^{-1} ),$$
    \RH{where $K$ denotes a $C^\infty$-smooth function.}
\end{remark}

We now write $L$ in the following form:
\begin{align*}
 L(r_1,\rho_{21})= J(\rho_{21})+K(r_1,r_1^{2(n-1)} \log r_1^{-1})-\frac{1}{2(n-1)}+\widetilde L(r_1,\rho_{21}),
\end{align*}
with $J$ and $K$ given above \RH{and $\widetilde L(0,0)=0$}. 
Inserting this into \eqref{r1Lrho21} gives
\begin{equation}\eqlab{widetildeLeqns}
\begin{aligned}
 \dot r_1 &=-r_1,\\
 \dot{\widetilde L} &= -2(n-1) \widetilde L +r_1\rho_{21} R_0(r_1,\rho_{21}),\\
 \dot \rho_{21}&=\rho_{21}.
\end{aligned}
\end{equation}
For this system, we use normal form theory, see e.g. \cite[Proposition 4.6]{ZR02}: Fix any $k\in \mathbb N$. Then there exists a locally defined $C^k$-smooth change of coordinates $(r_1,\widetilde L,\rho_{21})\mapsto (r_1,\widehat L,\rho_{21})$, which is $\mathcal O(r_1\rho_{21})$-close to the identity, such that \eqref{widetildeLeqns} becomes
\begin{equation}\eqlab{widehatLeqns}
\begin{aligned}
 \dot r_1 &=-r_1,\\
 \dot{\widehat L} &= -2(n-1) \widehat L+\kappa(r_1\rho_{21})r_1^{2(n-1)},\\
 \dot \rho_{21}&=\rho_{21},
\end{aligned}
\end{equation}
with $\kappa$ smooth, satisfying $\kappa(0)=0$.
Notice in particular in comparison with \cite[Proposition 4.6]{ZR02} that $\alpha_{i}=0$ for all $i\in \{0,\ldots,N(k)\}$ (using the notation of \cite{ZR02}) since these resonant terms are absent in \eqref{widetildeLeqns}. It is easy to see that 
\begin{align*}
 \widehat L(r_1,\rho_{21}) = {\kappa}(r_1\rho_{21}) r_1^{2(n-1)} \log r_1^{-1}, 
\end{align*}
defines an invariant manifold for \eqref{widehatLeqns}. We summarize our findings in the following Lemma.
\begin{lemma}
 Fix any $k\in \mathbb N$ and let $\delta>0$ be sufficiently small. Then there exists a $C^k$ smooth function $$\overline L:[0,\delta)\times [0,\delta^{2(n-1)} \log \delta^{-1})\times [0,\delta)\rightarrow \mathbb R$$
 satisfying
 \begin{align*}
  \overline L(r_1,r_1^{2(n-1)}\log r_1^{-1},0)&=r_1^{2(n-1)}\int_\delta^{r_1}\frac{v^{1-2n}}{\zeta_{2n}(v,0)}dv,\quad   \overline L(0,0,\rho_{21})=-\int_0^1\frac{v^{2n-3}}{Q_\lambda(\rho_{21} v)}dv,
 \end{align*}
 in particular, $\overline L(0,0,0)=\frac{1}{2(n-1)}$, 
so that 
\begin{align}
z_{21} \RH{=\Theta (r_1,Z,\rho_{21})}= \frac{1}{Z+\overline L(r_1,r_1^{2(n-1)}\log r_1^{-1},\rho_{21}) },\eqlab{ZZ}
\end{align}
defines a $C^1$-smooth change of coordinates 
that brings \eqref{xz21normal} into the (almost linearized) form:
\begin{align}\eqlab{final-flat}
 \dot r_1 &=-r_1,\nonumber\\
 \dot Z &=-2(n-1) Z 
 +R_1(r_1,\Theta (r_1,Z,\rho_{21}),r_1^{\RH{2n-1}}\rho_{21}\Theta (r_1,Z,\rho_{21}),\rho_{21}),\\
 \dot \rho_{21}&=\rho_{21},\nonumber
\end{align}
where $R_1$ is $C^\infty$-flat w.r.t. the third argument \RH{(see \eqref{decompo-n>1})}.  
\end{lemma}
\NEW{In other words, we have obtained a linearization \eqref{ZZ} of \eqref{xz21normal}, which is $C^k$-smooth w.r.t. $$(r_1,r_1^{2(n-1)}\log r_1^{-1},\rho_{21}),$$ up to $C^\infty$-flat terms (in the sense described in the lemma). }
Fix any $l\in \mathbb N$. Then upon increasing $k$, it is subsequently possible to remove $R_1$ {in \eqref{final-flat}} by a \NEW{subsequent} $C^l$-smooth change of coordinates $(r_1,\widehat Z,\rho_{21})\mapsto (r_1,Z,\rho_{21})$ with 
\begin{align*}
 Z = \widehat Z + \phi(r_1,\widehat Z,\rho_{21}),
\end{align*}
where $\phi(r_1,\widehat  Z,\rho_{21}) = \mathcal O((r_1 \rho_{21})^l)$, so that $$\dot{\widehat Z} = -2(n-1) \widehat Z.$$ This again follows from normal form theory (due to the absence of resonance terms, see e.g. \cite[Proposition 4.6]{ZR02}).

We now define the following sections 
\begin{align*}
\Sigma_{\text{in},21}\,&:\quad z_{21}\in I_{21},\,r_1=\delta,\,\rho_{21} \in [0,\nu), \\
\Sigma_{\text{out},21}\,&:\quad z_{21}\in I_{21}, \,r_1\in [0,\nu),\,\rho_{21}=\delta,
\end{align*}
with $I_{21}$ an (appropriate) closed interval that contains $2(n-1)\in I_{21}$, and where $\delta,\nu>0$ are small enough. \RH{After putting all the information together, we easily get} \KUK{the following.}
\begin{lemma}\label{lemma-very-important}
Fix any $k\in \mathbb N$, $k\gg 2(n-1)$, and $\delta>0$ small enough.
Then the transition map 
\begin{align*}
\Sigma_{\textnormal{in},21}&\rightarrow \Sigma_{\textnormal{out},21},\\
(\delta,z_{21},\rho_{21})&\mapsto (\rho_{21},z_{21+}(z_{21},\rho_{21}), \delta),
\end{align*}
given by the forward flow of \eqref{xz21normal}, is well-defined for all $\rho_{21}\in [0,\rho_{210})$ with $\rho_{210}>0$ small enough.
In particular, $z_{21+}$ takes the following from:
\begin{align*}
 z_{21+}(z_{21},\rho_{21}) = \frac{1}{\left(\frac{\rho_{21}}{\delta}\right)^{2(n-1)} \widehat Z_0+\overline L(\rho_{21},\rho_{21}^{2(n-1)}\log \rho_{21}^{-1},\delta)+ \phi(\rho_{21},\RH{\left(\frac{\rho_{21}}{\delta}\right)^{2(n-1)}}\widehat Z_0,\delta)},
\end{align*}
with $\phi\in C^k$, $\phi(r_{1},\widehat Z,\rho_{21})=\mathcal O((r_1\rho_{21})^k)$ uniformly w.r.t. $\widehat Z$, and
where $\widehat Z_0$ is a $C^k$-smooth function of $z_{21}$ and $\rho_{21}$ defined implicitly by
\begin{align*}
 z_{21} = \frac{1}{\widehat Z_0+\overline L(\delta,\delta^{2(n-1)}\log \delta^{-1},\rho_{21})+\phi(\delta,\widehat Z_0,\rho_{21})}
\end{align*}

\end{lemma}

\subsection{The chart $\bar \rho_2=1$} Consider again the system \eqref{equation-repeated}.
In the chart $\bar \rho_2=1$, we \RH{have $x=r_2x_2,z_2=r_2^{2(n-1)}z_{22},\rho_2=r_2$ and} obtain the desingularized vector-field defined by
\begin{equation}\eqlab{xz22}
\begin{aligned}
 \dot x_2 &=  P_\lambda(x_2) + x_2^{2n}\left(\zeta_{2n}(r_2x_2,r_2)+1\right)+r_2^{2n-3} z_{22}^2 G_{22}(r_2x_2,r_2^{2n-1}z_{22},r_2),\\
   \dot z_{22} &=-x_2 z_{22}^2,\\
   \dot r_2 &= 0,
\end{aligned}
\end{equation}
with $G_{22}(x,z,\epsilon)=z^{-2} \e^{-1/z} g(x,\e^{-1/z},\epsilon)$, 
after division of the right hand side by $r_2^{2n-1}$. Here $G_{22}$ is $C^\infty$-flat w.r.t. its second argument.

For $r_2=0$, we obtain 
\begin{align*}
 \dot x_2 &=P_\lambda(x_2),\\
 \dot z_{22}&=-x_2z_{22}^2,
\end{align*}
which by the assumption $P_\lambda(x_2)<0$ for all $x_2\in \mathbb R$, see \eqref{P2ncond}, is regular. By eliminating time, we have 
\begin{align*}
 -\frac{1}{z_{22}^2} \frac{dz_{22}}{d x_2} &= \frac{x_2}{P_\lambda(x_2)},
 \end{align*}
so that 
\begin{align*}
 z_{22}(x_{2})  = \frac{1}{z_{22}(x_{20})^{-1}+\int_{x_{20}}^{x_2} \frac{v}{P_\lambda(v)}dv},
\end{align*}
\RH{using an initial condition at $x_2=x_{20}$.}

\RH{Within the invariant plane $r_1=0$ of \eqref{xz21normal}, we have the following dynamics:
\begin{equation}\eqlab{xz21normal-reduced}
\begin{aligned}
   \dot z_{21} &=z_{21}\left(2(n-1)+\frac{z_{21}}{Q_\lambda(\rho_{21})}\right) ,\nonumber\\
   \dot \rho_{21} &= \rho_{21}.\nonumber
\end{aligned}
\end{equation}
It is not difficult to see that the (local) unstable manifold of this system at the hyperbolic saddle $(z_{21},\rho_{21})=(2(n-1),0)$ is the graph of 
$$z_{21} = J(\rho_{21})^{-1}, \ \rho_{21}\ge 0, $$
where $J$ is defined in \eqref{J-def}. Now, we have} \KUK{the following.}
\begin{lemma}\label{lemma-unstable-manifold}
 The unstable manifold $z_{21} = J(\rho_{21})^{-1},\rho_{21}\ge 0$, from the chart $\bar x=1$ takes the following form 
 \begin{align*}
 z_{22}(x_{2}) = \frac{1}{-\int_{x_{2}}^\infty \frac{v}{P_\lambda(v)}dv}
 \end{align*}
in the chart $\bar \rho_2=1$. It intersects $x_{2}=0$ in the point $z_{22} = \frac{1}{-\int_{0}^\infty \frac{v}{P_\lambda(v)}dv}$. 
\end{lemma}
\begin{proof}
\RH{ This follows directly using $Q_\lambda(\rho_{21})=\rho_{21}^{2n} P_\lambda(\rho_{21}^{-1})$, \eqref{J-def} and the change of coordinates in \eqref{cc12}.}
\end{proof}
We now define the transverse section $\Sigma_{\text{final},22}\,:\,x_2=0,z_{22}\in [0,\mu]$ for $\mu>0$ fixed and for all $0\le r_2\ll 1$. We then combine the previous results to obtain the following:
\begin{proposition}
 The transition map $(x_{\text{in}},\delta)\mapsto (0,z_{22+}(x_{\text{in}},\epsilon))$ from the original section $\Sigma_{\text{in}}\,:x\in I_{\text{in}},\,y=\delta$ to the section $\Sigma_{\text{final},22}$ is well-defined for all $0\le \epsilon\ll 1$. In particular,
 \begin{align*}
  z_{22+} =\frac{1}{-\int_{0}^\infty \frac{v}{P_\lambda(v)}dv} + \KUK{o(1)},
 \end{align*}
where $\KUK{o(1)}$ is $C^k$-smooth w.r.t. $(x_{\text{in}},\epsilon,\epsilon \log \epsilon^{-1})$ and is zero for $\epsilon=0$. 
\end{proposition}
\RH{\begin{proof}
    From Lemma \ref{lemma-boundary} and Section \ref{subsection-chart-bar-epsilon} it follows that the $z_{21}$-component of the transition map from the original section $\Sigma_{\text{in}}$ to the section $\Sigma_{\text{in},21}$ (defined before Lemma \ref{lemma-very-important}) is $C^k$-smooth w.r.t. $(x_{\text{in}},\epsilon,\epsilon \log \epsilon^{-1})$. This combined with Lemma \ref{lemma-very-important} implies that the $z_{21}$-component of the transition map from the section $\Sigma_{\text{in}}$ to the section $\Sigma_{\text{out},21}$ is again $C^k$-smooth w.r.t. $(x_{\text{in}},\epsilon,\epsilon \log \epsilon^{-1})$ and it is equal to $1/J(\delta)$ when $\epsilon=0$. Finally, using Lemma \ref{lemma-unstable-manifold} and the fact that the passage from $\Sigma_{\text{out},21}$ to $\Sigma_{\text{final},22}$ is regular, we obtain the property of $z_{22+}$.
\end{proof}

}

By replacing $x$ by $-x$ and $t$ by $-t$, we obtain \eqref{xy} with $P_\lambda(x)$ and $\zeta_{2n}(x,\epsilon)$ in \eqref{flambda} replaced by $P_\lambda(-x)$ and $\zeta_{2n}(-x,\epsilon)$, respectively. We therefore obtain the following expression for $z_{22-}$ in backward time:
\begin{align*}
  z_{22-} &=\frac{1}{-\int_{0}^\infty \frac{v}{P_\lambda(-v)}dv} + \KUK{o(1)} =\frac{1}{-\int_{0}^{-\infty} \frac{v}{P_\lambda(v)}dv} + \KUK{o(1)}.
 \end{align*}

\RH{Using the expressions for $z_{22+}$ and $z_{22-}$ and the change of coordinates $y=e^{-\frac{1}{\epsilon^{2n-1}z_{22}}}$ (see Section \ref{section-blow-up}), we finally get \eqref{zinzout}.}

Setting $z_{22+}=z_{22-}$ gives
\begin{align*}
 \int_{-\infty}^\infty \frac{v}{P_\lambda(v)}dv = \KUK{o(1)},
\end{align*}
where $\KUK{o(1)}$ is $C^k$-smooth w.r.t. \RH{$(x_{\text{in}},x_{\text{out}},\epsilon,\epsilon \log \epsilon^{-1})$} and is zero for $\epsilon=0$. \RH{Therefore, if \eqref{condI} holds true, then the entry-exit problem $I_{\text{in}}\ni x_{\text{in}}\mapsto \Delta(x_{\text{in}},\epsilon)\in I_{\text{out}}$ is not well-defined for all $0<\epsilon\ll 1$. This completes the proof of Theorem \ref{thm-n>1}.}


\section{Entry-exit problem at infinity in the Dumortier-Roussarie-Rousseau  program}
 \label{example}

In this section, 
we consider
\begin{equation}\eqlab{DDR-equ}
\begin{aligned}
 \dot x &= \epsilon\left(  \epsilon^2\lambda_0 + \epsilon\lambda_1 x+ x^{2} \zeta_{2}(x,\epsilon) \right)+ y \left(-1+\mathcal O(\epsilon)\right),\\
 \dot y &= - x y,
\end{aligned}
\end{equation}
with $\zeta_{2}(x,\epsilon)=-1+\beta x+\epsilon x^2\widetilde\zeta_{2}(x,\epsilon)$ where $\beta> 0$ and $\widetilde\zeta_{2}$ is a smooth function. Note that $\zeta_{2}(x,0)<0$ for all $x<\frac{1}{\beta}$. Clearly, system \eqref{DDR-equ} is a special case of \eqref{xy} with $n=1$ and $P_{\lambda}(x_2) =\lambda_0 +   \lambda_1 x_2  - x_2^{2}$. See also \eqref{flambda} and \eqref{P2n}. We assume that $4\lambda_0+\lambda_1^2<0$ (this implies that $P_{\lambda}$ is negative). 

\RH{Before we compute the entry-exit function associated with \eqref{DDR-equ}, let us briefly explain the connection between system \eqref{DDR-equ} and the graphics $(I_2^1)$ and $(I_4^1)$ in the Dumortier-Roussarie-Rousseau program (for more details see \cite{HuzRou202.}). 

We define a 5-parameter family of quadratic systems
\begin{equation}\eqlab{DDR-unfolding}
\begin{aligned}
 \dot x &= Ax-y+x^2+ (\mu_2+\mu_3)xy+\mu_1y^2,\\
 \dot y &= Cx +x^2+xy+\mu_3y^2,
\end{aligned}
\end{equation}
with $A$ close to $1$, $C$ close to $C_0>0$, and $\mu_1,\mu_2,\mu_3$ kept close to zero. When $A=1$, $C=C_0$ and $(\mu_1,\mu_2,\mu_3)=(0,0,0)$, the parabola $y=\frac{1}{2}x^2-\frac{C_0}{2}$ is invariant for system \eqref{DDR-unfolding}, and $(I_4^1)$ (resp. $(I_2^1)$) contains the parabola and nilpotent saddle-node at infinity and corresponds to $C_0=1$ (resp. $C_0>1$). The graphic $(I_4^1)$ contains a finite saddle-node located on the parabola. In contrast, the parabola is regular for $(I_2^1)$. (The case $C_0<1$ is not relevant since there can be no passage along the parabola.) } \RH{The full unfolding in quadratic systems of these graphics is given by \eqref{DDR-unfolding} (we refer to \cite{HuzRou202.}). 

If we apply the transformation $(x,y)=(\frac{v}{z},\frac{1}{z})$ (the positive $y$-direction) to system \eqref{DDR-unfolding}, we obtain
\begin{equation}\eqlab{DDR-unfolding-1}
\begin{aligned}
 \dot v &= \mu_1+\mu_2v-v^3-z+Avz-Cv^2z,\\
 \dot z &=-z\left(\mu_3+v+v^2+Cvz\right),
\end{aligned}
\end{equation}
after multiplying by $z$. Notice that the invariant line $z=0$ corresponds to $y=\infty$.  When $(\mu_1,\mu_2,\mu_3)=(0,0,0)$, system \eqref{DDR-unfolding-1} has a nilpotent saddle-node of multiplicity 4 at $(v,z)=(0,0)$. We make the rescaling $(\mu_1,\mu_2,\mu_3)=(\nu^2\bar\mu_1,\nu\bar\mu_2,\nu\bar\mu_3)$, with $(\bar\mu_1,\bar\mu_2,\bar\mu_3)\in\mathbb S^2$, $\nu>0$ small. Then we perform a blow-up of \eqref{DDR-unfolding-1} (with $\dot \nu=0$ augmented): $(v,z,\nu)=(r\bar v,r^2\bar z,r\bar\nu)$, with $(\bar v,\bar z,\bar\nu)\in\mathbb S^2$ and $r>0$. In the family chart $\bar\nu=1$, we obtain a desingularized system of slow-fast type, with the line of singularities $\bar z=0$ for $(\bar\mu_1,\bar\mu_2,r)=(0,0,0)$. We refer to Fig. \ref{fig-Motivation}. Finally, after additional scaling $(\bar\mu_1,\bar\mu_2, r)=(\epsilon\tilde\mu_1,\epsilon\tilde\mu_2,\epsilon\tilde r)$, with $(\tilde\mu_1,\tilde\mu_2,\tilde r)\in\mathbb S^2$ and $\epsilon> 0$, it is not difficult to see that this system (with the parameters kept in a suitable region) is analytically equivalent to a slow-fast system of type \eqref{DDR-equ}. (There also exists a region in the parameter space where the classical entry-exit problem \cite{DMS2016,Hsu19} occurs.)

The computation of the entry-exit function for \eqref{DDR-equ} plays a crucial role in detecting  limit periodic sets whose cyclicity
needs to be studied and in proving finite cyclicity of such sets.} \KUK{We postpone further details on the cyclicity to \cite{HuzRou202.}. }

The entry-exit formula \eqref{EE-remark} \NEW{for \eqref{DDR-equ}} becomes 
\begin{align*}
\beta\int_{x_{\text{out}}^b}^{x_{\text{in}}^b} \frac{1}{-1+\beta s}ds+ \log\left(-\frac{x_{\text{out}}^b}{x_{\text{in}}^b}\right)=\frac{\lambda_1 \pi}{\sqrt{-4\lambda_0-\lambda_1^2}},
\end{align*}
or, equivalently,
\begin{align*}
\log\left(-\frac{(1-\beta x_{\text{in}}^b)x_{\text{out}}^b}{(1-\beta x_{\text{out}}^b)x_{\text{in}}^b}\right)=\frac{\lambda_1 \pi}{\sqrt{-4\lambda_0-\lambda_1^2}},
\end{align*}
with $x_{\text{in}}^b\in (0,\frac{1}{\beta})$ and $x_{\text{out}}^b<0$.
Using this, we get
\begin{align}\eqlab{entry-exit-DDR}
 x_{\text{out}}^b=\frac{\e^{K} x_{\text{in}}^b }{\beta\left(\e^K+1\right)x_{\text{in}}^b-1},
\end{align}
where we write $K=\frac{\lambda_1 \pi}{\sqrt{-4\lambda_0-\lambda_1^2}}$. Since we require $x_{\text{out}}^b<0$, from \eqref{entry-exit-DDR} it follows that $x_{\text{in}}^b$ has to be kept in the interval $ (0,\frac{1}{\beta\left(\e^K+1\right)})$.

When $\epsilon=0$, system \eqref{DDR-equ} becomes 
\begin{equation}\eqlab{DDR-equ-1}
\begin{aligned}
 \dot x &= -y,\\
 \dot y &= - x y.
\end{aligned}
\end{equation}
The fast fibers of \eqref{DDR-equ-1} are parabolas $y=\frac{1}{2}x^2+C$. The parabola passing through $(x_{\text{in}},\delta)$ intersects the $x$-axis at the base point $(x_{\text{in}}^b,0)$ with 
\begin{align}\eqlab{base-point-in}
    x_{\text{in}}^b=\sqrt{x_{\text{in}}^2-2\delta}.
\end{align}
Similarly, the base point $(x_{\text{out}}^b,0)$ of $(x_{\text{out}},\delta)$ is given by
\begin{align}\eqlab{base-point-out}
    x_{\text{out}}^b=-\sqrt{x_{\text{out}}^2-2\delta}.
\end{align}
If we plug the expressions \eqref{base-point-in} and \eqref{base-point-out} into the formula \eqref{entry-exit-DDR}, we finally get 
\begin{align}\eqlab{entry-exit-DDR-final}
 x_{\text{out}}=\Delta_0(x_{\text{in}}):=-\sqrt{2\delta+ \frac{\e^{2K}(x_{\text{in}}^2-2\delta)}{\left(\beta\left(\e^K+1\right)\sqrt{x_{\text{in}}^2-2\delta}-1\right)^2}}.
\end{align}
We suppose that $I_{\text{in}}$ is a segment with 
$$I_{\text{in}}\subset\left(\sqrt{2\delta},\sqrt{2\delta +\frac{1}{\beta^2(\e^K+1)^2}} \right),$$
and $I_{\text{out}}\subset(-\infty,0)$ is an appropriate segment. 
Then the following result is a simple consequence of Theorem \ref{Thm-n=1}.
\begin{theorem}\label{Thm-DDR} Fix any $k\in \mathbb N$.
Then the Dulac map  $\Delta(\cdot,\epsilon):I_{\text{in}}\rightarrow I_{\text{out}}$ associated with \eqref{DDR-equ} is well-defined for all $\epsilon\in ]0,\epsilon_0[$ and takes the following form
\begin{align*}
 \Delta(x_{\text{in}},\epsilon) =\Delta_0(x_{\text{in}})+\phi(x_{\text{in}},\epsilon,\epsilon \log \epsilon^{-1}),
\end{align*}
where $\Delta_0(x_{\text{in}})$ is defined in \eqref{entry-exit-DDR-final} and $\phi$ is $C^k$-smooth and satisfies $\phi(x_{\text{in}},0,0)=0$ for all $x_{\text{in}}\in I_{\text{in}}$. 
\end{theorem}

\subsection{\KUK{Numerical computations of $\Delta$ for \eqref{DDR-equ}}}
In Fig. \ref{fig-DDR}(a), we have used Matlab's ODE15s with low tolerances ($10^{-12}$) to compute $x_{\text{out}}=\Delta(x_{\text{in}},\epsilon)$ for \eqref{DDR-equ} with $\mathcal O(\epsilon)$ and $\widetilde \zeta_2$ both set to zero, and the following parameter values:
\begin{align}
   \lambda_0 = -2,\quad \lambda_1 =1,\quad \beta = 1,\quad \delta=\frac12,\eqlab{DDR-para}
\end{align}
and $\epsilon=0.01$ (magenta), $\epsilon=0.005$ (red) and $\epsilon=0.001$ (blue). The result is  clearly in agreement with Theorem \ref{Thm-DDR} as we see a convergence towards the theoretical curve $x_{\text{out}} = \Delta_0(x_{\text{in}})$ given by \eqref{entry-exit-DDR-final} (black and dashed). In order to compute $x_{\text{out}}$, it was crucial to use the corresponding $(x,z)$-system:
\begin{equation}\nonumber
\begin{aligned}
 \dot x &= \epsilon\left(  \epsilon^2\lambda_0 + \epsilon\lambda_1 x+ x^{2} (-1+\beta x) \right)- \e^{-1/z},\\
 \dot z &= - x z^2.
\end{aligned}
\end{equation}
Indeed, in the $(x,y)$-coordinates, $y$ becomes exponentially small w.r.t. $\epsilon\to 0$ and round-off errors lead to meaningless predictions of $x_{\text{out}}$ (without any significant delay; results of this are not shown for simplicity). \textit{We therefore speculate that the transformation \eqref{yz} has practical significance for numerical computations of entry-exit problems in general}. 

In Fig. \ref{fig-DDR}(b), we show trajectories in the $(x,z_2)$-plane for fixed $x_{\text{in}}=1.016$ ($x_{\text{in}}^b = 0.18$), recall the definition of $z_2=\epsilon^{-1}z$ in \eqref{z2here}, \NEW{and $\epsilon=0.01$ (magenta), $\epsilon=0.005$ (red) and $\epsilon=0.001$ (blue)}. The theoretical curve for $\epsilon=0$ given by \eqref{z2xb} is shown in dashed and black for comparison. Notice the cusp-like behavior of the trajectories as they pass close to $x=0$ (which is a degenerate line for $\epsilon=0$, recall \NEW{\figref{new}} and \figref{blowup1}). This is due to the $1/\log \vert x\vert$-behavior of $z_2$ near $x=0$, recall the discussion below \eqref{z2xb}. 
\begin{figure}[htb]
	\begin{center}
		\subfigure[]{\includegraphics[width=.485\textwidth]{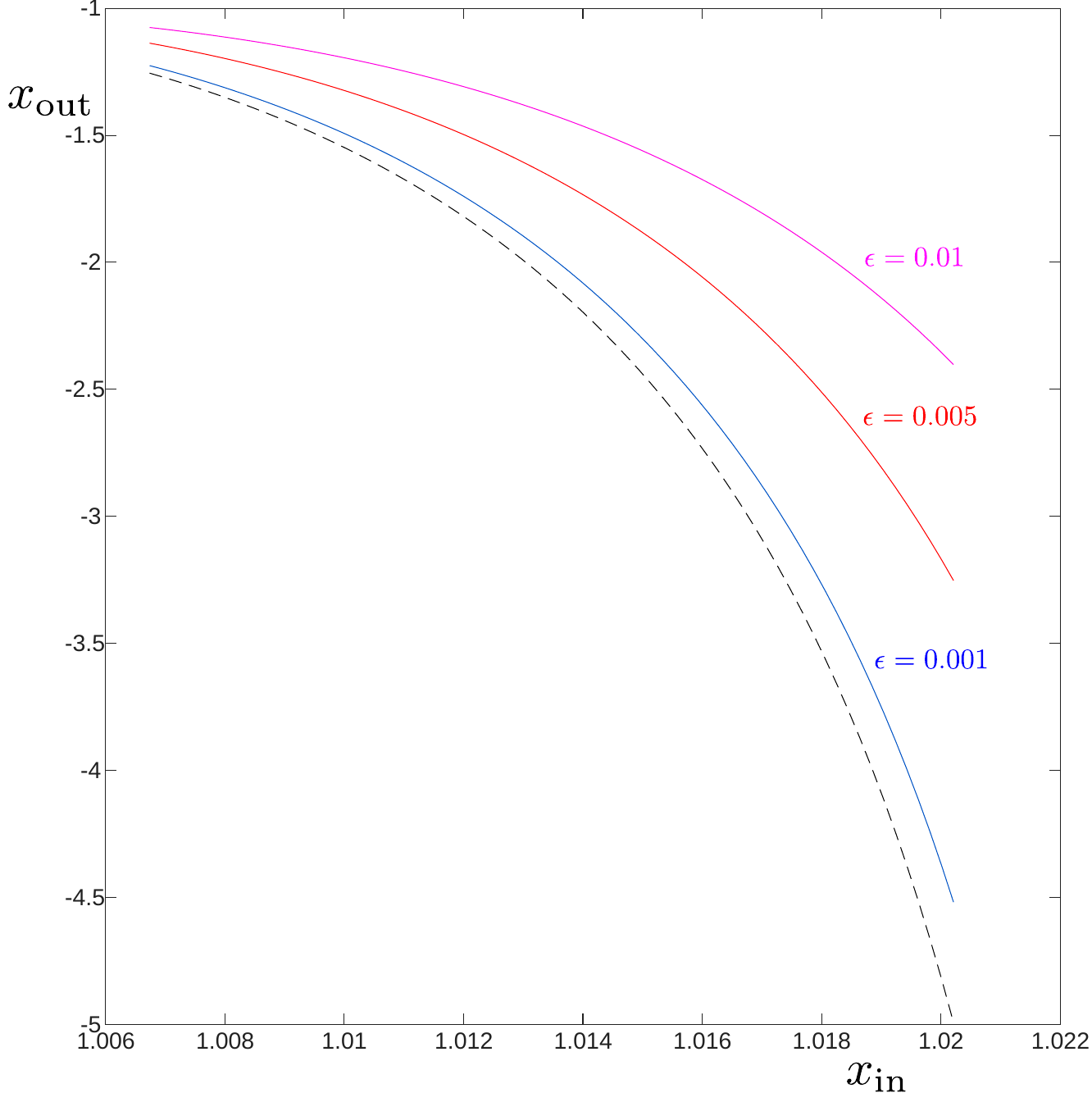}}
        \subfigure[]{\includegraphics[width=.49\textwidth]{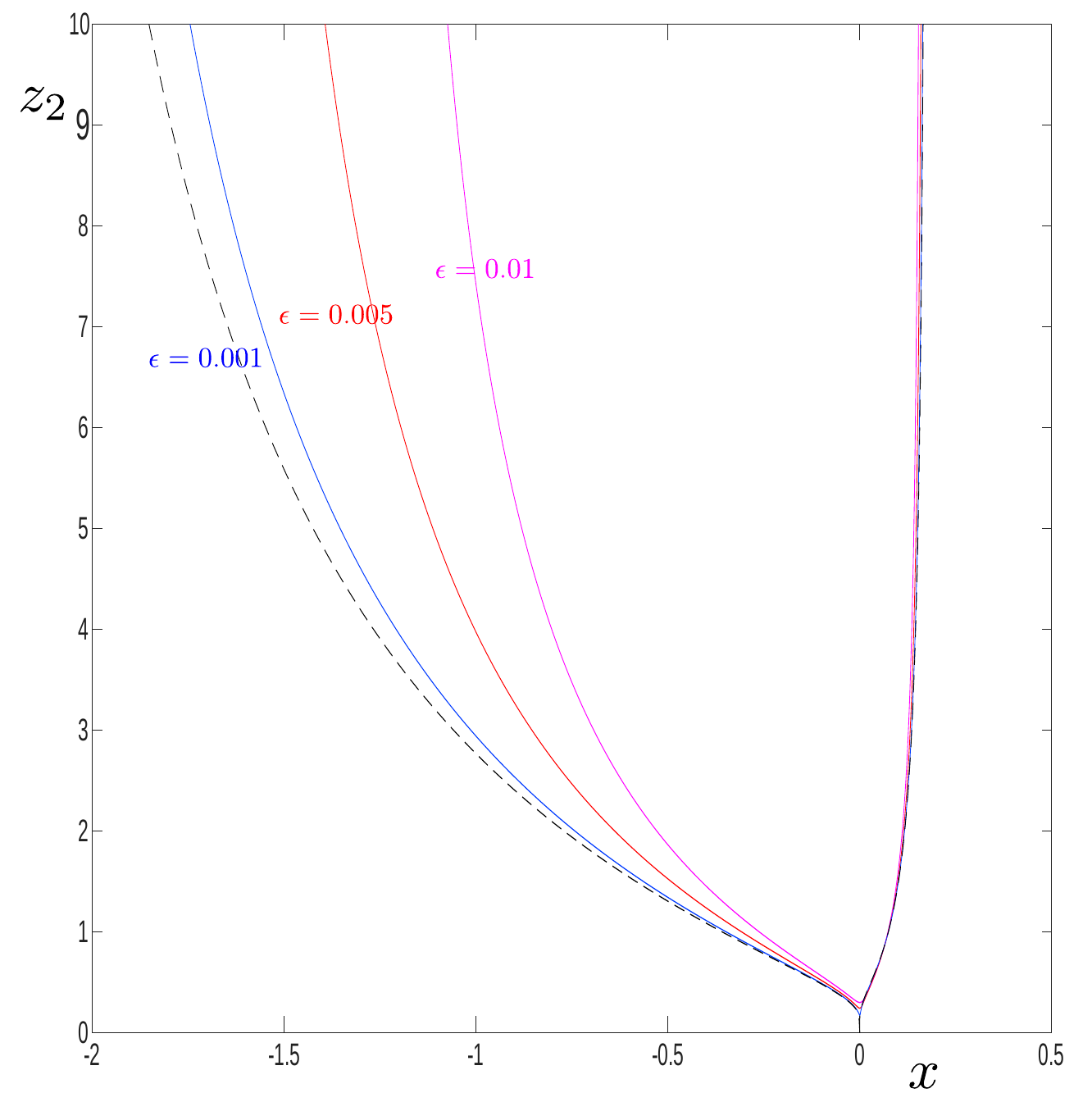}}
         \end{center}
 \caption{In (a): $x_{\text{out}}$ for \eqref{DDR-equ} with the parameter values \eqref{DDR-para} and $\epsilon=0.01$ (magenta), $\epsilon=0.005$ (red) and $\epsilon=0.001$ (blue), computed using Matlab's ODE15s (on the $(x,z)$-system) with low tolerances. The dashed black line is the theoretical curve obtained from \eqref{entry-exit-DDR-final}. In (b): Trajectories in the $(x,z_2)$-plane, recall \eqref{z2here}, for the same parameter values as in (a) but with $x_{\text{in}}=1.016$ ($x_{\text{in}}^b = 0.18$) fixed. The dashed black line is again the theoretical curve obtained from \eqref{z2xb}.}
	\label{fig-DDR}
\end{figure}

\section*{Acknowledgments} 
The authors thank Technical University of Denmark for the hospitality during R. Huzak's research visit in the spring of 2025. R. Huzak's stay was facilitated by K. U. Kristiansen's  Danish Research Council (DFF) grant 4283-00014B. Finally, the research of R. Huzak was supported by Croatian Science Foundation (HRZZ) grant IP-2022-10-9820.

\newpage

\bibliography{refs}
\bibliographystyle{plain}
\end{document}